\documentclass[graybox]{svmult}
\usepackage{amsmath,amssymb}
\usepackage{paralist}
\usepackage{bbm}
\usepackage{xcolor}

\usepackage{mathptmx}       
\usepackage{helvet}         
\usepackage{courier}        
\usepackage{type1cm}        
\usepackage{makeidx}         
\usepackage{graphicx}        
\usepackage{multicol}        
\usepackage[bottom]{footmisc}

\makeindex             

\newcommand{\R}{\mathbb{R}}

\begin{document}

\title*{Herglotz' generalized variational principle and contact type Hamilton-Jacobi equations}
\titlerunning{Herglotz' generalized variational principle}
\author{Piermarco Cannarsa, Wei Cheng, Kaizhi Wang and Jun Yan}
\institute{Piermarco Cannarsa \at Dipartimento di Matematica, Universit\`a di Roma ``Tor Vergata'', Via della Ricerca Scientifica 1, 00133 Roma, Italy, \email{cannarsa@mat.uniroma2.it}
\and Wei Cheng \at Department of Mathematics, Nanjing University, Nanjing 210093, China, \email{chengwei@nju.edu.cn} \and Kaizhi Wang \at School of Mathematical Sciences, Shanghai Jiao Tong University, Shanghai 200240, China, \email{kzwang@sjtu.edu.cn}
\and Jun Yan \at School of Mathematical Sciences, Fudan University and Shanghai Key Laboratory for Contemporary Applied Mathematics, Shanghai 200433, China, \email{yanjun@fudan.edu.cn}}
%
%
\maketitle

\abstract{We develop an approach for the analysis of fundamental solutions to Hamilton-Jacobi equations of contact type based on a generalized variational principle proposed by Gustav Herglotz. We also give a quantitative Lipschitz estimate on the associated minimizers.}

\section{Introduction}
The so called {\em generalized variational principle} was proposed by Gustav Herglotz in 1930 (see \cite{Herglotz1} and \cite{Herglotz2}).  It generalizes classical variational principle by defining the functional, whose extrema are sought, by a differential equation. More precisely, the functional $u$ is defined in an implicit way by an ordinary differential equation
\begin{equation}\label{eq:Herglotz_intro}
	\dot{u}(s)=F(s,\xi(s),\dot{\xi}(s),u(s)),\quad s\in[0,t],
\end{equation}
with $u(t)=u_0\in\R$, for $t>0$, a function $F\in C^2(\R\times\R^n\times\R^n\times\R,\R)$ and a piecewise $C^1$ curve $\xi:[0,t]\to\R^n$. Here, $u=u[\xi,s]$ can be regarded as a functional, on a space of paths $\xi(\cdot)$. The generalized variational principle of Herglotz is as follows:

\medskip

{\em Let the functional $u=u[\xi,t]$ be defined by \eqref{eq:Herglotz_intro} with $\xi$ in the space of piecewise $C^1$ functions on $[0,t]$. Then the value of the functional $u[\xi,t]$ is an extremal for the function $\xi$ such that the variation $\frac d{d\varepsilon}u[\xi+\varepsilon\eta,t]=0$ for arbitrary piecewise $C^1$ function $\eta$ such that $\eta(0)=\eta(t)=0$.}


\medskip

Herglotz reached the idea of the generalized variational principle through his work on contact transformations and their connections with Hamiltonian systems and Poisson brackets. His work was motivated by ideas from S. Lie, C. Carath\'eodory and other researchers.  An important reference on the generalized variational principle is the monograph \cite{GGG}. The variational principle of Herglotz is important for many reasons:
\begin{itemize}[--]
  \item The solutions of the equations \eqref{eq:Herglotz_intro} determine a family of contact transformations, see \cite{GGG,Caratheodory,Eisenhart,Giaquinta-Hildebrandt};
  \item  The generalized variational principle gives a variational description of energy-nonconservative processes even when $F$ in \eqref{eq:Herglotz_intro} is independent of $t$.
  \item If $F$ has the form $F=-\lambda u+L(x,v)$, then the relevant problems are closely connected to the Hamilton-Jacobi equations with discount factors (see, for instance, \cite{DFIZ,DFIZ2,Cannarsa-Quincampoix,Ishii-Mitake-Tran1,Ishii-Mitake-Tran2,Maro-Sorrentino,Gomes2008,Iturriaga-Sanchez-Morgado}). As an extension to nonlinear discounted problems, various examples are discussed in \cite{Chen-Cheng-Zhao,Zhao-Cheng}.
  \item Even for a energy-nonconservative process which can be described with the generalized variational principle, one can systematically derive conserved quantities as Noether's theorems such as \cite{GG1,GG2};
  \item The generalized variational principle provides a link between the mathematical structure of control and optimal control theories and contact transformation (see \cite{Furuta-Sano});
  \item There are some interesting connections between contact transformations and  equilibrium thermodynamics (see, for instance, \cite{Mrugala}).
\end{itemize}

In this note, we will clarify more connections between the generalized variational principle of Herglotz and Hamilton-Jacobi theory motivated by recent works in \cite{Wang-Wang-Yan1,Wang-Wang-Yan2} under a set of Tonelli-like conditions. We will begin with generalized variational principle of Herglotz in the frame of Lagrangian formalism  different from the methods used in \cite{Wang-Wang-Yan1,Wang-Wang-Yan2}. Throughout this paper, let $L:\R^n\times\R\times\R^n$ be a function of class $C^2$ such that the following standing assumptions are satisfied:
\begin{enumerate}[(L1)]
  \item $L(x,r,\cdot)>0$ is strictly convex for all $(x,r)\in \R^n\times\R$.
  \item There exist two superlinear nondecreasing function $\overline{\theta}_0,\theta_0:[0,+\infty)\to[0,+\infty)$, $\theta_0(0)=0$ and $c_0>0$, such that
  $$
  \overline{\theta}_0(|v|)\geqslant L(x,0,v)\geqslant\theta_0(|v|)-c_0,\quad (x,v)\in\R^n\times\R^n.
  $$
  \item There exists $K>0$ such that
  $$
  |L_r(x,r,v)|\leqslant K,\quad (x,r,v)\in \R^n\times\R\times\R^n.
  $$
\end{enumerate}

\begin{remark}\label{rem:condition}
	For each $r\in \R$,	from the conditions (L2) and (L3) we could take
	$$
	\overline{\theta}_r:=\overline{\theta}_0+K|r|, \quad\theta_r:=\theta_0,  \quad c_r:=c_0+K|r|,
	$$
	such that
	\begin{equation}\label{sup_lin}
  	    \overline{\theta}_r(|v|)\geqslant L(x,r,v)\geqslant\theta_r(|v|)-c_r,\quad (x,v)\in\R^n\times\R^n.
  	\end{equation}
  	Obviously, $\overline{\theta}_r$ and  $\theta_r$  are both nonnegative, superlinear and nondecreasing functions, $c_r>0$.
\end{remark}

It is natural to introduce the associated Hamiltonian
$$
H(x,r,p)=\sup_{v\in\R^n}\{\langle p,v\rangle-L(x,r,v)\},\quad (x,r,p)\in \R^n\times\R\times(\R^n)^*.
$$

Let $x,y\in\R^n$, $t>0$ and $u_0\in\R$. Set
\begin{align*}
	\Gamma^t_{x,y}=\{\xi\in W^{1,1}([0,t],\R^n): \xi(0)=x,\ \xi(t)=y\}.
\end{align*}
We consider a variational problem
\begin{equation}\label{eq:M}
	\text{Minimize}\quad u_0+\inf\int^t_0L(\xi(s),u_{\xi}(s),\dot{\xi}(s))\ ds,
\end{equation}
where the infimum is taken over all $\xi\in\Gamma^t_{x,y}$ such that the Carath\'eodory equation
\begin{equation}\label{eq:caratheodory_L_intro}
	\dot{u}_{\xi}(s)=L(\xi(s),u_{\xi}(s),\dot{\xi}(s)),\quad a.e.\ s\in[0,t],
\end{equation}
admits an absolutely continuous solution $u_{\xi}$ with initial condition $u_{\xi}(0)=u_0$. It is already known that the variational problem \eqref{eq:M} with subsidiary conditions \eqref{eq:caratheodory_L_intro} is closely connected to the Hamilton-Jacobi equations in the form
\begin{equation}\label{eq:static_intro}
	H(x,u(x),Du(x))=c.
\end{equation}
The readers can refer to \cite{Giaquinta-Hildebrandt} for a systematic approach of Hamilton-Jacobi equations in the form \eqref{eq:static_intro} especially in the context of contact geometry.

In \cite{Wang-Wang-Yan1,Wang-Wang-Yan2}, a weak KAM type theory on equations \eqref{eq:static_intro} was developed on compact manifolds under the aforementioned Tonelli-like conditions. Problem \eqref{eq:M} is understood as an implicit variational principle (\cite{Wang-Wang-Yan1}) and, by introducing the positive and negative Lax-Oleinik semi-groups, an existence result for weak KAM type solutions of \eqref{eq:static_intro} was obtained provided $c$ in the right side of equation \eqref{eq:static_intro} belongs to the set of critical values (\cite{Wang-Wang-Yan2}). The same approach adapts to the evolutionary equations in the form
\begin{equation}\label{eq:evolutionary_intro}
	D_tu+H(x,u,D_xu)=0.
\end{equation}

Unlike the methods used in \cite{Wang-Wang-Yan1,Wang-Wang-Yan2}, in this note, our approach of the equations  \eqref{eq:static_intro} and \eqref{eq:evolutionary_intro} is based on the the variational problem \eqref{eq:M} under subsidiary conditions \eqref{eq:caratheodory_L_intro}. We give all the details of such a Tonelli-like theory and its connection to viscosity solutions of \eqref{eq:static_intro} and \eqref{eq:evolutionary_intro} .

In view of Proposition \ref{existence} below, the infimum in \eqref{eq:M} can be achieved. Suppose that $\xi\in\Gamma^t_{x,y}$ is a minimizer for \eqref{eq:M} where $u_{\xi}$ is uniquely determined by \eqref{eq:caratheodory_L_intro} with initial condition $u_{\xi}(0)=u_0$. Then we call such $\xi$ an {\em extremal}. Due to Proposition \ref{Herglotz_contact} below, each extremal $\xi$ and associated $u_{\xi}$ are of class $C^2$ and satisfy the Herglotz equation (Generalized Euler-Lagrange equation by Herglotz)
\begin{equation}\label{eq:Herglotz1}
	\begin{split}
		&\,\frac d{ds}L_v(\xi(s),u_{\xi}(s),\dot{\xi}(s))\\
	=&\,L_x(\xi(s),u_{\xi}(s),\dot{\xi}(s))+L_u(\xi(s),u_{\xi}(s),\dot{\xi}(s))L_v(\xi(s),u_{\xi}(s),\dot{\xi}(s)).
	\end{split}
\end{equation}
Moreover, let $p(s)=L_v(\xi(s),u_{\xi}(s),\dot{\xi}(s))$ be the so called dual arc. Then $p$ is also of class $C^2$ and we conclude that $(\xi,p,u_{\xi})$ satisfies the following Lie equation
\begin{equation}\label{eq:contact}
\begin{cases}
  		\dot{\xi}(s)=H_p(\xi(s),u_{\xi}(s),p(s));\\
  		\dot{p}(s)=-H_x(\xi(s),u_{\xi}(s),p(s))-H_u(\xi(s),u_{\xi}(s),p(s))p(s);\\
  		\dot{u}_{\xi}(s)=p(s)\cdot\dot{\xi}(s)-H(\xi(s),u_{\xi}(s),p(s)),
\end{cases}
\end{equation}
where the reader will recognize the classical system of characteristics for \eqref{eq:static_intro}.

The paper is organized as follows: In Section 2, we afford a detailed and rigorous treatment of \eqref{eq:M} under subsidiary conditions \eqref{eq:caratheodory_L_intro}. In Section 3, we study the regularity of the minimizers and deduce the Herglotz equation \eqref{eq:Herglotz1} and Lie equation \eqref{eq:contact} as well. In Section 4, we show that the two approaches between \cite{Wang-Wang-Yan1,Wang-Wang-Yan2} and ours are equivalent. We also sketch the way to move Herglotz' variational principle to manifolds.

\section{Existence of minimizers in Herglotz' variational principle}

Fix $x_0,x\in\R^n$, $t>0$ and $u_0\in\R$. Let $\xi\in\Gamma^t_{x_0,x}$, we consider the Carath\'eodory equation
\begin{equation}\label{eq:app_caratheodory_L}
	\begin{cases}
		\dot{u}_{\xi}(s)=L(\xi(s),u_{\xi}(s),\dot{\xi}(s)),\quad a.e.\ s\in[0,t],&\\
		u_{\xi}(0)=u_0.&
	\end{cases}
\end{equation}
We define the action functional
\begin{equation}\label{eq:app_fundamental_solution}
	J(\xi):=\int^t_0L(\xi(s),u_{\xi}(s),\dot{\xi}(s))\ ds,
\end{equation}
where $\xi\in\Gamma^t_{x_0,x}$ and $u_{\xi}$ is defined in \eqref{eq:app_caratheodory_L} by Proposition \ref{caratheodory} in Appendix. Notice that Carath\'eodory's theorem (Proposition \ref{caratheodory}) is just a local result, but the existence and uniqueness of the solution of \eqref{eq:app_caratheodory_L} holds on $[0,t]$ since condition (L3) and that $\xi\in\mathcal{A}$. Our purpose is to minimize $J(\xi)$ over
\begin{align*}
	\mathcal{A}=\{\xi\in\Gamma^t_{x_0,x}: \text{\eqref{eq:app_caratheodory_L} admits an absolutely continuous solution $u_{\xi}$}\}.
\end{align*}
Notice that $\mathcal{A}\not=\varnothing$ because it contains all piecewise $C^1$ curves connecting $x_0$ to $x$. It is not hard to check that, for each $a\in\R$,
\begin{align*}
	\mathcal{A}=\mathcal{A}':=\{\xi\in\Gamma^t_{x_0,x}: \text{the function $s\mapsto L(\xi(s),a,\dot{\xi}(s))$ belongs to $L^1([0,t])$}\}.
\end{align*}
Indeed, If $\xi\in\mathcal{A}$, then $L(\xi(s),u_{\xi}(s),\dot{\xi}(s))$ is integrable on $[0,t]$ and $u_{\xi}$ is bounded. Thus $\xi\in\mathcal{A}'$ since
\begin{align*}
	|L(\xi,0,\dot{\xi})|\leqslant|L(\xi,u_{\xi},\dot{\xi})|+K|u_{\xi}|.
\end{align*}
On the other hand, if $\xi\in\mathcal{A}'$, then
\begin{align*}
	\dot{u}_{\xi}\leqslant L(\xi,0,\dot{\xi})+K|u_{\xi}|.
\end{align*}
Therefore, $\xi\in\mathcal{A}$.

For the following estimate, we define $L_0(x,v):=L(x,0,v)$.

\begin{lemma}\label{u_low_bound}
	Let $x_0,x\in\R^n$, $t>0$, $u_0\in\R$. Given $\xi\in\Gamma^t_{x_0,x}$ such that \eqref{eq:app_caratheodory_L} admits an absolutely continuous solution, then we have that
	\begin{equation}\label{eq:u_low_bound2}
		|u_{\xi}(s)|\leqslant\exp(Ks)(|u_0|+c_0s)
	\end{equation}
	if $u_{\xi}(s)<0$. In particular, we have
	\begin{equation}\label{eq:u_low_bound}
		u_{\xi}(s)\geqslant-\exp(Ks)(|u_0|+c_0s),\quad s\in[0,t].
	\end{equation}
\end{lemma}

\begin{proof}
	Let $x_0,x\in\R^n$, $t>0$, $u_0\in\R$ and $\xi\in\mathcal{A}$. Suppose that $u_{\xi}(s_0)<0$, $s_0\in(0,t]$. We define $E=\{s\in[0,s_0): u_{\xi}(s)\geqslant 0\}$ and
	\begin{align*}
		a=\begin{cases}
			0& E=\varnothing, \\
			\sup E & E\not=\varnothing.
		\end{cases}
	\end{align*}
	Then, we have that $u_{\xi}(s)\leqslant0$ for all $s\in[a,s_0]$ and $u_{\xi}(a)=0$ if $E\not=\varnothing$. Now, we are assuming that $E\not=\varnothing$. For any $s\in[a,s_0]$ we have that
	\begin{align*}
		-|u_{\xi}(s)|=&\,u_{\xi}(s)=u_{\xi}(a)+\int^s_aL(\xi(\tau),u_{\xi}(\tau),\dot{\xi}(\tau))\ d\tau\\
		\geqslant&\,-|u_{\xi}(a)|+\int^s_aL_0(\xi(\tau),\dot{\xi}(\tau))\ d\tau-K\int^s_a|u_{\xi}(\tau)|\ d\tau\ d\tau\\
		\geqslant&\,-|u_{\xi}(a)|+\int^s_a\theta_0(|\dot{\xi}(\tau)|)\ d\tau-c_0(s-a)-K\int^s_a|u_{\xi}(\tau)|\ d\tau\\
		\geqslant&\,-|u_{\xi}(a)|-c_0s-K\int^s_a|u_{\xi}(\tau)|\ d\tau.
	\end{align*}
	Then, we have that
	\begin{align*}
		|u_{\xi}(s)|\leqslant (|u_0|+c_0s)+K\int^s_a|u_{\xi}(\tau)|\ d\tau,\quad s\in[a,s_0].
	\end{align*}
	Then Gronwall inequality implies
	\begin{align*}
		|u_{\xi}(s)|\leqslant\exp(K(s-a))(|u_0|+c_0s)\leqslant\exp(Ks)(|u_0|+c_0s),\quad s\in[a,s_0].
	\end{align*}
	
	If $E=\varnothing$, then $a=0$ and the proof is the same. This leads to \eqref{eq:u_low_bound2} and \eqref{eq:u_low_bound}.\hfill\mbox{\qed}
\end{proof}

In view to Lemma \ref{u_low_bound}, we conclude that $\inf_{\xi\in\mathcal{A}}J(\xi)$ is bounded below. Now, for any $\varepsilon>0$, set
\begin{equation*}
	\mathcal{A}_{\varepsilon}=\{\xi\in\mathcal{A}: \inf_{\eta\in\mathcal{A}}J(\eta)+\varepsilon\geqslant u_{\xi}(t)-u_0\}.
\end{equation*}

\begin{lemma}\label{bound_u(t)}
	Suppose $x_0\in\R^n$, $t,R>0$, $u_0\in\R$ and $|x-x_0|\leqslant R$. Let $\varepsilon>0$ and $\xi\in\mathcal{A}_{\varepsilon}$. Then we have that
	\begin{align*}
		u_{\xi}(t)-u_0\leqslant t(\kappa(R/t)+K|u_0|)\exp(Kt)+\varepsilon,
	\end{align*}
	with $\kappa(r)=\overline{\theta}_0(r)+2c_0$. Moreover, there exist two nondecreasing and superlinear functions $F,G:[0,+\infty)\to[0,+\infty)$ such that
	\begin{equation}
		|u_{\xi}(t)|\leqslant tF(R/t)+G(t)|u_0|+\varepsilon,
	\end{equation}
	where $F(r)=\max\{\kappa(r),c_0\exp(Kr)\}$ and $G(r)=\max\{rK\exp(Kr)+1,\exp(Kr)\}$.
\end{lemma}

\begin{proof}
	Suppose $x_0\in\R^n$, $t,R>0$, $u_0\in\R$ and $|x-x_0|\leqslant R$. Let $\varepsilon>0$ and $\xi\in\mathcal{A}_{\varepsilon}$. First, notice that
	\begin{equation}\label{eq:kappa}
		|L_0(x,v)|\leqslant L_0(x,v)+2c_0\leqslant\overline{\theta}_0(|v|)+2c_0,\quad (x,v)\in\R^n\times\R^n.
	\end{equation}
	Set $\kappa(r)=\overline{\theta}_0(r)+2c_0$.
	
	Define $\xi_0(s)=x_0+s(x-x_0)/t$ for any $s\in[0,t]$, then $\xi_0\in\mathcal{A}$. Then, for any $s\in[0,t]$, we have that
	\begin{align*}
		|u_{\xi_0}(s)-u_0|\leqslant&\,\int^s_0|L_0(\xi_0,\dot{\xi}_0)|\ d\tau+K\int^s_0|u_{\xi_0}|\ d\tau\\
		\leqslant&\,t\kappa(R/t)+K\int^s_0|u_{\xi_0}-u_0|\ d\tau+tK|u_0|.
	\end{align*}
	Due to Gronwall inequality, we obtain
	\begin{equation}\label{eq:straight_line}
		|u_{\xi_0}(s)-u_0|\leqslant t(\kappa(R/t)+K|u_0|)\exp(Kt),\quad s\in[0,t].
	\end{equation}
	Together with Lemma \ref{u_low_bound}, this completes the proof.\hfill\mbox{\qed}
\end{proof}

\begin{lemma}\label{bound_u}
	Suppose $x_0\in\R^n$, $t,R>0$, $u_0\in\R$ and $|x-x_0|\leqslant R$. Let $\varepsilon>0$ and $\xi\in\mathcal{A}_{\varepsilon}$. Then there exist two continuous functions $F_1, F_2:[0,+\infty)\times[0,+\infty)\to[0,+\infty)$ depending on $R$, with $F_i(r_1,\cdot)$ being nondecreasing and superlinear and $F_i(\cdot,r_2)$ being nondecreasing for any $r_1,r_2\geqslant0$, $i=1,2$, such that
	\begin{equation}\label{eq:u_bound}
		|u_{\xi}(s)|\leqslant tF_1(t,R/t)+C_1(t)(\varepsilon+|u_0|),\quad s\in[0,t]
	\end{equation}
	and
	\begin{equation}\label{eq:L1}
		\int^t_0|L(\xi,u_{\xi},\dot{\xi})|\ d\tau\leqslant tF_2(t,R/t)+C_2(t)(\varepsilon+|u_0|),
	\end{equation}
	where $C_i(t)>0$ for $i=1,2$.
\end{lemma}

\begin{proof}
	Suppose $x_0\in\R^n$, $t,R>0$, $u_0\in\R$ and $|x-x_0|\leqslant R$. Let $\varepsilon>0$ and $\xi\in\mathcal{A}_{\varepsilon}$.
	
	If $u_{\xi}(t)\geqslant 0$, we define $E_+=\{s\in[0,t]: u_{\xi}(s)>u_{\xi}(t)\}$. If $E_+=\varnothing$, then we have that $u_{\xi}(s)\leqslant u_{\xi}(t)$ for all $s\in[0,t]$. Now, we suppose that $E_+\not=\varnothing$. It is known that $E_+$ is the union of a countable family of open intervals $\{(a_i,b_i)\}$ which are mutually disjoint (It is possible that $a_i=0$ and this case can be dealt with separately but similarly). For any $\tau\in E_+$, there exists an open interval $(a,b)$, a component of $E_+$ containing $s$, such that $u_{\xi}(\tau)>u_{\xi}(t)\geqslant0$ for all $\tau\in(a,b)$ and $u_{\xi}(b)=u_{\xi}(t)$. Therefore, for almost all $s\in[a,b]$, we have that
	\begin{align*}
		\dot{u}_{\xi}(s)=L(\xi(s),u_{\xi}(s),\dot{\xi}(s))\geqslant L_0(\xi(s),\dot{\xi}(s))-Ku_{\xi}(s).
	\end{align*}
	Invoking condition (L2), it follows that, for all $s\in[a,b]$,
	\begin{align*}
		e^{Kb}u_{\xi}(b)-e^{Ks}u_{\xi}(s)\geqslant\int^b_se^{K\tau}L_0(\xi(\tau),\dot{\xi}(\tau))\ d\tau\geqslant-c_0(b-s)e^{Kb}
	\end{align*}
	Thus we obtain that
	\begin{equation}\label{eq:upper1}
		\begin{split}
			u_{\xi}(s)\leqslant&\,c_0(b-s)e^{K(b-s)}+e^{K(b-s)}u_{\xi}(t)\\
			\leqslant&\,c_0te^{Kt}+e^{Kt}[(t\kappa(R/t)+K|u_0|)e^{Kt}+\varepsilon+|u_0|]\\
			=&\,tF_1(t,R/t)+G_1(t)|u_0|+\varepsilon,
		\end{split}
		\quad s\in [0,t],
	\end{equation}
	where $F_1(r_1,r_2):=e^{Kr_1}(c_0+\kappa(r_2))$ and $G_1(r)=e^{Kr}(Ke^{Kr}+1)$.
	
	If $u_{\xi}(t)<0$, define $v_{\xi}(s)=u_{\xi}(s)-u_{\xi}(t)$, then $v_{\xi}(s)$ satisfies the Carath\'eodory eqaution
	\begin{align*}
		\dot{v}_{\xi}(s)=L(\xi(s),v_{\xi}(s)+u_{\xi}(t),\dot{\xi}(s)),\quad s\in[0,t]
	\end{align*}
	with initial condition $v_{\xi}(0)=u_0-u_{\xi}(t)$. Similarly, We define $F_+=\{s\in[0,t]: v_{\xi}(s)>v_{\xi}(t)\}$. If $F_+=\varnothing$, then we have that $v_{\xi}(s)\leqslant v_{\xi}(t)=0$ for all $s\in[0,t]$. Now, we suppose that $F_+\not=\varnothing$ and $F_+$ is the union of a countable family of open intervals $\{(c_i,d_i)\}$ which are mutually disjoint. For any $s\in F_+$, there exists an open interval, say $(c,d)$, such that $v_{\xi}(s)>v_{\xi}(t)=0$ for all $s\in(c,d)$ and $v_{\xi}(d)=v_{\xi}(t)$. Therefore, for almost all $s\in[c,d]$, we have that
	\begin{align*}
		\dot{v}_{\xi}(s)\geqslant L_0(\xi(s),\dot{\xi}(s))-Kv_{\xi}(s)-K|u_{\xi}(t)|.
	\end{align*}
	It follows that, for all $s\in[c,d]$,
	\begin{align*}
		e^{Kd}v_{\xi}(d)-e^{Ks}v_{\xi}(s)\geqslant&\,\int^d_se^{K\tau}L_0(\xi(s),\dot{\xi}(s))\ d\tau-Kt|u_{\xi}(t)|e^{Kt}\\
		\geqslant&\,-c_0te^{Kd}-Kt|u_{\xi}(t)|e^{Kt},
	\end{align*}
	and this gives rise to
	\begin{align*}
		v_{\xi}(s)\leqslant c_0te^{K(d-s)}+K|u_{\xi}(t)|te^{K(t-s)}+e^{K(d-s)}v_{\xi}(d)\leqslant c_0te^{Kt}+Kt|u_{\xi}(t)|e^{Kt},
	\end{align*}
	since $v_{\xi}(d)=0$. It follows that, for all $s\in[0,t]$,
	\begin{equation}\label{eq:upper2}
		\begin{split}
			u_{\xi}(s)\leqslant &\,c_0te^{Kt}+Kte^{Kt}|u_{\xi}(t)|+u_{\xi}(t)\leqslant c_0te^{Kt}+(Kte^{Kt}+1)|u_{\xi}(t)|\\
		\leqslant&\,c_0te^{Kt}+(Kte^{Kt}+1)(tF_2(R/t)+G_2(t)|u_0|+\varepsilon)
		\end{split}
	\end{equation}
	with $F_2$ and $G_2$ determined by Lemma \ref{bound_u(t)}. By combining \eqref{eq:upper1} and \eqref{eq:upper2} and setting
	\begin{gather*}
		F_3(r_1,r_2)=\max\{F_1(r_1,r_2),c_0e^{Kr_1}+F_2(r_2)(Kr_1e^{Kr_1}+1)\},\\
		C_1(r)=\max\{G_1(r),G_2(t)(Kr_1e^{Kr_1}+1)\},\quad C_2(r)=\max\{C_1(r),e^{Kr}c_0\},
	\end{gather*}
	we conclude that
	\begin{align}
		u_{\xi}(s)\leqslant&\, tF_3(t,R/t)+C_1(t)(|u_0|+\varepsilon),\label{eq:upper3}\\
			|u_{\xi}(s)|\leqslant&\, tF_3(t,R/t)+C_2(t)(|u_0|+\varepsilon).\label{eq:upper4}
	\end{align}
	This leads to the proof of \eqref{eq:u_bound} together with Lemma \ref{u_low_bound}.
	
	Now, by \eqref{eq:kappa}, Lemma \ref{bound_u(t)} and \eqref{eq:upper4}, we have that
	\begin{align*}
		\int^s_0|L_0(\xi,\dot{\xi})|d\tau\leqslant&\,\int^s_0(L_0(\xi,\dot{\xi})+2c_0)\ d\tau\leqslant 2c_0s+u_{\xi}(s)-u_0+K\int^s_0|u_{\xi}|\ d\tau\\
		\leqslant&\, 2c_0t+tF_2(t,R/t)+C_2(t)(|u_0|+\varepsilon)+|u_0|\\
		&\,+t^2KF_2(t,R/t)+tKC_2(t)(|u_0|+\varepsilon)\\
		\leqslant&\,tF_4(t,R/t)+C_3(t)(|u_0|+\varepsilon).
	\end{align*}
	Therefore, \eqref{eq:L1} follows from the estimate below
	\begin{equation*}
		\begin{split}
			&\,\int^t_0|L(\xi,u_{\xi},\dot{\xi})|\ d\tau\leqslant\int^t_0|L_0(\xi,\dot{\xi})|\ d\tau+K\int^t_0|u_{\xi}|\ d\tau\\
			\leqslant&\,tF_4(t,R/t)+C_3(t)(|u_0|+\varepsilon)+tK(tF_3(t,R/t)+C_2(t)(|u_0|+\varepsilon))\\
			=&\,tF_5(t,R/t)+C_4(t)(|u_0|+\varepsilon).
		\end{split}
	\end{equation*}
	We relabel the function $F_i$ and this completes our proof.\hfill\mbox{\qed}
\end{proof}

\begin{remark}\label{validate}
	Now, fix any $\varepsilon\in(0,1)$ and any $\xi\in\mathcal{A}_{\varepsilon}\subset\mathcal{A}_1$. The definition of \eqref{eq:app_fundamental_solution} can be replaced by $L_{\phi}=L(x,\phi(u),v)$ with $\phi:\R\to\R$ a bounded nondecreasing smooth function such that $\phi(u)=u$ for $|u|\leqslant tF_1(t,R/t)+C_1(t)(1+|u_0|)$ and $\phi(u)\equiv u^*$, a suitably selected real number, for $|u|\geqslant tF_1(t,R/t)+C_1(t)(1+|u_0|)+1$, where $F_1(t,R/t)$ and $C_1(t)$ are determined by \eqref{eq:u_bound} in Lemma \ref{bound_u} and $F_1$ and $C_1$ are both independent of $\varepsilon$. Therefore, to minimize $J$ defined in \eqref{eq:app_fundamental_solution}, we can suppose that $\sup_{\xi\in\mathcal{A}_1}|L(\xi(s),u,\dot{\xi}(s))|$ is bounded by an integrable function $f\in L^1([0,t])$. In this situation, \eqref{eq:app_caratheodory_L} is indeed a Carath\'eodory equation which admits a unique solution by Proposition \ref{caratheodory}.
\end{remark}

\begin{corollary}\label{equi_int_u}
For any $\varepsilon\in(0,1)$, the set $\{u_{\xi}:\xi\in\mathcal{A}_{\varepsilon}\}$ is relatively compact in $C^0([0,t],\R)$.
\end{corollary}

\begin{proof}
    Suppose $x_0\in\R^n$, $t,R>0$, $u_0\in\R$ and $|x-x_0|\leqslant R$. For any $\varepsilon\in(0,1)$ and $\xi\in\mathcal{A}_{\varepsilon}$. Recall that $u_{\xi}$ is the unique solution of \eqref{eq:app_caratheodory_L} by Remark \ref{validate}, it follows $\{\dot{u}_{\xi}\}_{\xi\in\mathcal{A}_{\varepsilon}}$ is equi-integrable which implies $\{u_{\xi}\}_{\xi\in\mathcal{A}_{\varepsilon}}$ is equi-continuous. The boundedness of $\{u_{\xi}\}_{\xi\in\mathcal{A}_{\varepsilon}}$ follows from Lemma \ref{bound_u}. Invoking Ascoli-Arzela theorem, we get our conclusion.\hfill\mbox{\qed}
\end{proof}

\begin{lemma}\label{equi_integrable}
	Suppose $x_0\in\R^n$, $t,R>0$, $u_0\in\R$ and $|x-x_0|\leqslant R$. Let $\varepsilon\in(0,1)$ and $\xi\in\mathcal{A}_{\varepsilon}$. Then there exist a continuous function $F=F_{u_0,R}:[0,+\infty)\times[0,+\infty)\to[0,+\infty)$, $F(r_1,\cdot)$ is nondecreasing and superlinear and $F(\cdot,r_2)$ is nondecreasing for any $r_1,r_2\geqslant0$, such that
	\begin{align*}
		\int^t_0|\dot{\xi}(s)|\ ds\leqslant tF(t,R/t)+\varepsilon.
	\end{align*}
	Moreover, the family $\{\dot{\xi}\}_{\xi\in\mathcal{A}_{\varepsilon}}$ is equi-integrable.
\end{lemma}

\begin{proof}
 Let $\varepsilon>0$ and $\xi\in\mathcal{A}_{\varepsilon}$. Then, by (L2) we obtain
	\begin{equation}\label{eq:lower1}
		\begin{split}
			&\,u_{\xi}(t)-u_0\\
			=&\,\int^t_0L(\xi(s),u_{\xi}(s),\dot{\xi}(s))\ ds\geqslant\int^t_0\{L(\xi(s),0,\dot{\xi}(s))-K|u_{\xi}(s)|\}\ ds\\
		\geqslant&\, \int^t_0\{\theta_0(|\dot{\xi}(s)|)-c_0-K|u_{\xi}(s)|\}\ ds\\
		\geqslant&\,\int^t_0\{|\dot{\xi}(s)|-K|u_{\xi}(s)|-(c_0+\theta^*_0(1))\}\ ds.
		\end{split}
	\end{equation}
	In view of Lemma \ref{bound_u(t)}, Lemma \ref{bound_u} and \eqref{eq:lower1}, we obtain that
	\begin{align*}
		\int^t_0|\dot{\xi}(s)|\ ds\leqslant&\,\int^t_0K|u_{\xi}(s)|\ ds+t(c_0+\theta^*_0(1))+u_{\xi}(t)-u_0\\
		\leqslant&\,tK(tF_1(t,R/t)+C_1(t)(\varepsilon+|u_0|))+t(c_0+\theta^*_0(1))\\
		&\,+tF_2(t,R/t)+\varepsilon:=tF_3(t,R/t)+\varepsilon.
	\end{align*}
	
	Now we turn to proof of the equi-integrability of the family $\{\dot{\xi}\}_{\xi\in\mathcal{A}_{\varepsilon}}$. Since $\theta_0$ is a superlinear function, then for any $\alpha>0$ there exists $C_{\alpha}>0$ such that $r\leqslant\theta_0(r)/\alpha$ for $r>C_{\alpha}$. Thus, for any measurable subset $E\subset[0,t]$, invoking (L2), (L3) and Lemma \ref{bound_u}, we have that
	\begin{align*}
		\int_{E\cap\{|\dot{\xi}|>C_{\alpha}\}}|\dot{\xi}|ds\leqslant&\,\frac 1{\alpha}\int_{E\cap\{|\dot{\xi}|>C_{\alpha}\}}\theta_0(|\dot{\xi}|)ds\leqslant\frac 1{\alpha}\int_{E\cap\{|\dot{\xi}|>C_{\alpha}\}}(L_0(\xi,\dot{\xi})+c_0)ds\\
		\leqslant&\,\frac 1{\alpha}\int_{E\cap\{|\dot{\xi}|>C_{\alpha}\}}(L(\xi,u_{\xi},\dot{\xi})+K|u_{\xi}|+c_0)ds\\
		\leqslant&\,\frac 1{\alpha}(u_{\xi}(t)-u_0+K(tF_1(t,R/t)+C_1(t)(\varepsilon+|u_0|))+tc_0)\\
		\leqslant&\,\frac 1{\alpha}(tF_2(t,R/t)+1+K(tF_1(t,R/t)+C_1(t)(1+|u_0|))+tc_0)\\
		:=&\,\frac 1{\alpha}F_4(t,R/t).
	\end{align*}
	Therefore, we conclude that
	\begin{align*}
		\int_E|\dot{\xi}|ds\leqslant\int_{E\cap\{|\dot{\xi}|>C_{\alpha}\}}|\dot{\xi}|ds+\int_{E\cap\{|\dot{\xi}|\leqslant C_{\alpha}\}}|\dot{\xi}|ds\leqslant\frac 1{\alpha}F_4(t,R/t)+|E|C_{\alpha}.
	\end{align*}
	Then, the equi-integrability of the family $\{\dot{\xi}\}_{\xi\in\mathcal{A}_{\varepsilon}}$ follows since the right-hand side can be made arbitrarily small by choosing $\alpha$ large and $|E|$ small, and this proves our claim.
\end{proof}

\begin{proposition}\label{existence}
	The functional
	\begin{align*}
		\Gamma^t_{x_0,x}\ni\xi\mapsto J(\xi)=\int^t_0L(\xi(s),u_{\xi}(s),\dot{\xi}(s))\ ds,
	\end{align*}
	where $u_{\xi}$ is determined by \eqref{eq:app_caratheodory_L} with initial condition $u_{\xi}(0)=u_0$, admits a minimizer in $\Gamma^t_{x_0,x}$.
\end{proposition}

\begin{proof}
	Fix $x_0,x\in\R^n$, $t>0$ and $u_0\in\R$. Consider any minimizing sequence $\{\xi_k\}$ for $J$, that is, a sequence such that $J(\xi_k)\to\inf\{J(\xi):\xi\in\Gamma^t_{x_0,x}\}$ as $k\to\infty$. We want to show that this sequence admits a cluster point which is the required minimizer. Notice there exists an associated sequence $\{u_{\xi_k}\}$ given by \eqref{eq:app_caratheodory_L} in the definition of $J(\xi_k)$. The idea of the proof is standard but a little bit different.
	
	First, notice that Lemma \ref{equi_integrable} implies that the sequence of derivatives $\{\dot{\xi}_k\}$ is equi-integrable. Since the sequence $\{\dot{\xi}_k\}$ is equi-integrable, by the Dunford-Pettis Theorem there exists a subsequence, which we still denote by $\{\dot{\xi}_k\}$, and a function $\eta^*\in L^1([0,t],\R^n)$ such that $\dot{\xi}_k\rightharpoonup\eta^*$ in the weak-$L^1$ topology. The equi-integrability of $\{\dot{\xi}_k\}$ implies that the sequence $\{\xi_k\}$ is equi-continuous and uniformly bounded. Invoking Ascoli-Arzela theorem, we can also assume that the sequence $\{\xi_k\}$ converges uniformly to some absolutely continuous function $\xi_{\infty}\in\Gamma^t_{x_0,x}$. For any test function $\varphi\in C^1_0([0,t],\R^n)$,
	\begin{align*}
		\int^t_0\varphi\eta^* ds=\lim_{k\to\infty}\int^t_0\varphi\dot{\xi}_k ds=-\lim_{k\to\infty}\int^t_0\dot{\varphi}\xi_k ds=-\int^t_0\dot{\varphi}\xi_{\infty} ds.
	\end{align*}
	By the fundamental lemma in calculus of variation (see, for instance, \cite[Lemma 6.1.1]{Cannarsa-Sinestrari}), we can conclude that $\dot{\xi}_{\infty}=\eta^*$ almost everywhere. Similarly, Corollary \ref{equi_int_u} implies $\{u_{\xi_k}\}$ is relatively compact in $C^0([0,t],\R)$. Therefore $\{u_{\xi_k}\}$ converges uniformly to $u_{\xi_{\infty}}$ by taking a subsequence if necessary.
	
	We recall a classical result (see, for instance, \cite[Theorem 3.6]{Buttazzo-Giaquinta-Hildebrandt} or \cite[Section 3.4]{Buttazzo}) on the sequentially lower semicontinuous property on the functional
	\begin{align*}
		L^1([0,t],\R^m)\times L^1([0,t],\R^n)\ni (\alpha,\beta)\mapsto\mathbf{F}(\alpha,\beta):=\int^t_0\mathbf{L}(\alpha(s),\beta(s))\ ds.
	\end{align*}
	One has that if
	\begin{inparaenum}[(i)]
	\item $\mathbf{L}$ is lower semicontinuous;
	\item $\mathbf{L}(\alpha,\cdot)$ is convex on $\R^n$,
	\end{inparaenum}
	then the functional $\mathbf{F}$ is sequentially lower semicontinuous on the space $L^1([0,t],\R^m)\times L^1([0,t],\R^n)$ endowed with the strong topology on $L^1([0,t],\R^m)$ and the weak topology on $L^1([0,t],\R^n)$. Now, let $\mathbf{L}(\alpha_{\xi_k}(s),\beta_{\xi_k}(s)):=L(\xi_k(s),u_{\xi_k}(s),\dot{\xi_k}(s))$ with $\alpha_{\xi_k}(s)=(\xi_k(s),u_{\xi_k}(s))$ and $\beta_{\xi_k}(s)=\dot{\xi_k}(s)$, then we have
	\begin{align*}
		\liminf_{k\to\infty}\int^t_0L(\xi_k(s),u_{\xi_k}(s),\dot{\xi_k}(s))\ ds\geqslant\int^t_0L(\xi_{\infty}(s),u_{\xi_{\infty}}(s),\dot{\xi}_{\infty}(s))\ ds.
	\end{align*}
	Therefore, $\xi_{\infty}\in\Gamma^t_{x_0,x}$ is a minimizer of $J$ and this completes the proof of the existence result.
\end{proof}

\begin{corollary}\label{bound_1}
	Let $u_0\in\R$  and $R>0$ be fixed. Then there exists  a continuous function $F=F_{u_0,R}:[0,+\infty)\times[0,+\infty)\to[0,+\infty)$,
	with $F(t,r)$ nondecreasing in both variables and superlinear with respect to $r$, such that for any given $t>0$ and $x,x_0\in\R^n$, with 
	$|x-x_0|\leqslant R$, every minimizer  $\xi\in\Gamma^t_{x_0,x}$ for \eqref{eq:app_fundamental_solution} satisfies 
	\begin{align*}
		\int^t_0|\dot{\xi}(s)|\ ds\leqslant tF(t,R/t)
	\end{align*}
and
	\begin{align*}
		\operatorname*{ess\ inf}_{s\in[0,t]}|\dot{\xi}(s)|\leqslant F(t,R/t),\quad\sup_{s\in[0,t]}|\xi(s)-x_0|\leqslant tF(t,R/t).
	\end{align*}
\end{corollary}

\begin{proof}
	The first assertion is direct from Lemma \ref{equi_integrable}. The last two inequality follows from the relations
	\begin{align*}
		\operatorname*{ess\ inf}_{s\in[0,t]}|\dot{\xi}(s)|\leqslant\frac 1t\int^t_0|\dot{\xi}(s)|\ ds,\quad\text{and}\quad |\xi(s)-x_0|\leqslant\int^t_0|\dot{\xi}(s)|\ ds,
	\end{align*}
	together with the first assertion.
\end{proof}
\section{Necessary conditions and regularity of minimizers}
\subsection{Lipschitz estimate of minimizers}
In order to study the regularity of a minimizers $\xi$ of \eqref{eq:app_fundamental_solution}, we need an {\em a priori} Lipschitz estimate for $\xi$. A key point of the proof of such an estimate is  the following {\em Erdmann condition}, which is standard for classical autonomous Tonelli Lagrangians. For the results in and after this section, we suppose the following technical condition
\begin{enumerate}[(L2')]
	\item $L$ satisfies condition (L2). Moreover there exists $R>1$ and $C_A>0$ depending on compact $A\subset\R^n$ such that
	\begin{align*}
		L(x,0,rv)\leqslant C_A(1+L(x,0.v)),\quad \forall r\in[1,R],\ (x,v)\in A\times\R^n.
	\end{align*}
\end{enumerate}

We begin we some fundamental results from convex analysis.

\begin{lemma}\label{convexity}
Let $L$ satisfy conditions \mbox{\rm (L1)-(L3)}. We conclude that
\begin{enumerate}[\rm (a)]
  \item The function
  \begin{equation}\label{eq:f_epsilon}
	f(\varepsilon):=L_v(x,r,v/(1+\varepsilon))\cdot v/(1+\varepsilon)-L(x,r,v/(1+\varepsilon))
  \end{equation}
  is decreasing for $\varepsilon>-1$. In particular, $f(\varepsilon)\geqslant f(+\infty)=-L(x,r,0)\geqslant-\overline{\theta}_0(0)-K|r|$.
  \item If $\varepsilon_1,\varepsilon_2>-1$ and $\varepsilon_1<\varepsilon_2$, then we have
  \begin{align*}
  	L(x,r,v/(1+\varepsilon_2))\leqslant (\kappa+1)^{-1}L(x,r,v/(1+\varepsilon_1))+\kappa\cdot(\kappa+1)^{-1}(\overline{\theta}_0(0)+K|r|)
  \end{align*}
  and
  \begin{align*}
  	f(\varepsilon_2)\leqslant\kappa^{-1}L(x,r,v/(1+\varepsilon_1))-(\kappa^{-1}+1)L(x,r,v/(1+\varepsilon_2))
  \end{align*}
  where $\kappa=(\varepsilon_2-\varepsilon_1)/(1+\varepsilon_1)>0$.
\end{enumerate}

\end{lemma}

\begin{proof}
Let $\varepsilon_1,\varepsilon_2\in(-1,+\infty)$ and $\varepsilon_1<\varepsilon_2$. We set $L(v)=L(x,r,v)$. By (L1) we have that
\begin{align*}
	L(v/(1+\varepsilon_2))\geqslant L(v/(1+\varepsilon_1))+L_v(v/(1+\varepsilon_1))\cdot\{v/(1+\varepsilon_2)-v/(1+\varepsilon_1)\}.
\end{align*}
It follows that
\begin{align*}
	f(\varepsilon_1)-f(\varepsilon_2)\geqslant&\,L_v(v/(1+\varepsilon_1))\cdot(v/(1+\varepsilon_1))-L_v(v/(1+\varepsilon_2))\cdot(v/(1+\varepsilon_2))\\
	&\,+L_v(v/(1+\varepsilon_1))\cdot\{v/(1+\varepsilon_2)-v/(1+\varepsilon_1)\}\\
	=&\,L_v(v/(1+\varepsilon_1))\cdot(v/(1+\varepsilon_2))-L_v(v/(1+\varepsilon_2))\cdot(v/(1+\varepsilon_2))\\
	=&\,\{L_v(v/(1+\varepsilon_1))-L_v(v/(1+\varepsilon_2))\}\cdot\{v/(1+\varepsilon_1)-v/(1+\varepsilon_2)\}\\
	&\,\cdot(1/(1+\varepsilon_1)-1/(1+\varepsilon_2))^{-1}\cdot(1/(1+\varepsilon_1))\\
	\geqslant&\,0.
\end{align*}
The last statement is direct together with (L2) and (L3).

Now we turn to the proof of (b). Observe that by convexity
\begin{align*}
	L(v/(1+\varepsilon_1))\geqslant&\,L(v/(1+\varepsilon_2))+L_v(v/(1+\varepsilon_2))\cdot\{v/(1+\varepsilon_1)-v/(1+\varepsilon_2)\}\\
	=&\,L(v/(1+\varepsilon_2))+\kappa\cdot L_v(v/(1+\varepsilon_2))\cdot(v/(1+\varepsilon_2)).
\end{align*}
In view of (a) we obtain that
\begin{align*}
	&\,L(v/(1+\varepsilon_1))-(\kappa+1)L(v/(1+\varepsilon_2))\\
	\geqslant&\,\kappa\cdot\{-L(v/(1+\varepsilon_2))+L_v(v/(1+\varepsilon_2))\cdot(v/(1+\varepsilon_2))\}\\
	\geqslant&\,-\kappa\cdot(\overline{\theta}_0(0)+K|r|)
\end{align*}
Then the first assertion follows. Moreover, we have that
\begin{align*}
	L_v(v/(1+\varepsilon_2))\cdot(v/(1+\varepsilon_2))\leqslant\kappa^{-1}(L(v/(1+\varepsilon_1))-L(v/(1+\varepsilon_2)))
\end{align*}
which leads to the second assertion.
\end{proof}

\begin{lemma}[Erdmann condition]\label{Erdmann_condition}
	Suppose \mbox{\rm (L1), (L2')} and \mbox{\rm (L3)} are satisfied. Let $\xi\in\Gamma^t_{x_0,x}$ be a minimizer of \eqref{eq:app_fundamental_solution} with $u_{\xi}$ determined by \eqref{eq:app_caratheodory_L} and $u_{\xi}(0)=u_0$. Set $\int^s_0L_udr=\int^s_0L_u(\xi(r),u_{\xi}(r),\dot{\xi}(r))dr$ and define
	\begin{align*}
		E(s)=e^{-\int^s_0L_udr}\cdot\left\{L_v(\xi(s),u_{\xi}(s),\dot{\xi}(s))\cdot\dot{\xi}(s)-L(\xi(s),u_{\xi}(s),\dot{\xi}(s))\right\}
	\end{align*}
	for almost all $s\in[0,t]$. Then $E$ has a continuous representation $\bar{E}$ such that $\bar{E}$ is absolutely continuous on $[0,t]$ and $\bar{E}'=0$ for almost all $s\in[0,t]$.
\end{lemma}

\begin{remark}
Condition (L2') is satisfied when $L$ has polynomial growth with respect to $v$. It is a key point to ensure the finiteness of the action after reparametrization.
\end{remark}

\begin{proof}
We divide the proof into several steps.

\medskip

\noindent\textbf{Step I: Reparametrization.} For any measurable function  $\alpha:[0,t]\to[1/2,3/2]$ satisfying $\int^t_0\alpha(s)\ ds=t$ (the set of all such functions $\alpha$ is denoted by $\Omega$), we define
\begin{align*}
	\tau(s)=\int^s_0\alpha(r)\ dr,\quad s\in [0,t].
\end{align*}
Note that $\tau:[0,t]\to[0,t]$ is a bi-Lipschitz map and its inverse $s(\tau)$ satisfies
$$
s'(\tau)=\frac 1{\alpha(s(\tau))},\quad a.e.\ \tau\in[0,t].
$$

Now, given $\xi\in\Gamma^t_{x_0,x}$  as above and $\alpha\in\Omega$, define the reparameterization $\eta$ of $\xi$ by $\eta(\tau)=\xi(s(\tau))$. It follows that $\dot{\eta}(\tau)=\dot{\xi}(s(\tau))/\alpha(s(\tau))$. Let $u_{\eta}$ be the unique solution of \eqref{eq:app_caratheodory_L} with initial condition $u_{\eta}(0)=u_0$. Then we have that
\begin{align*}
	J(\xi)\leqslant\, J(\eta)=&\int^t_0L(\eta(\tau),u_{\eta}(\tau),\dot{\eta}(\tau))\ d\tau\\
	=&\,\int^t_0L(\xi(s),u_{\xi,\alpha}(s),\dot{\xi}(s)/\alpha(s))\alpha(s)\ ds
\end{align*}
where $u_{\xi,\alpha}$ solves
\begin{align*}
	\dot{u}_{\xi,\alpha}(s)=L(\xi(s),u_{\xi,\alpha}(s),\dot{\xi}(s)/\alpha(s))\alpha(s),\quad u_{\xi,\alpha}(0)=u_0.
\end{align*}
By a direct calculation, for all $\alpha\in\Omega$ and almost all $s\in[0,t]$, we obtain 
\begin{align*}
	\dot{u}_{\xi,\alpha}-\dot{u}_{\xi}=&\,L(\xi,u_{\xi,\alpha},\dot{\xi}/\alpha)\alpha-L(\xi,u_{\xi},\dot{\xi})\\
	=&\,L(\xi,u_{\xi,\alpha},\dot{\xi}/\alpha)\alpha-L(\xi,u_{\xi},\dot{\xi}/\alpha)\alpha+L(\xi,u_{\xi},\dot{\xi}/\alpha)\alpha-L(\xi,u_{\xi},\dot{\xi})\\
	=&\,\widehat{L_u^{\alpha}}\,(u_{\xi,\alpha}-u_{\xi})+L(\xi,u_{\xi},\dot{\xi}/\alpha)\alpha-L(\xi,u_{\xi},\dot{\xi}),
\end{align*}
and $u_{\xi,\alpha}(0)-u_{\xi}(0)=0$, where
\begin{align*}
	\widehat{L_u^{\alpha}}(s)=\int^1_0L_u\big(\xi(s),u_{\xi}(s)+\lambda(u_{\xi,\alpha}(s)-u_{\xi}(s)),\dot{\xi}(s)/\alpha(s)\big)\alpha(s)\ d\lambda.
\end{align*}
By solving the Carath\'eodory equation above, we conclude that
\begin{equation}\label{eq_implicit}
	u_{\xi,\alpha}(s)-u_{\xi}(s)=\int^s_0e^{\int^s_\tau\widehat{L_u^{\alpha}}dr}(L(\xi,u_{\xi},\dot{\xi}/\alpha)\alpha-L(\xi,u_{\xi},\dot{\xi}))\ d\tau,
\end{equation}
and $u_{\xi,\alpha}(t)-u_{\xi}(t)\geqslant0$ for all $\alpha\in\Omega$. We claim that 
\begin{equation}\label{eq:integrability}
	L(\xi,u_{\xi},\dot{\xi}/\alpha)\in L^1([0,t])\ \text{for all}\ \alpha\in\Omega.
\end{equation}
To show \eqref{eq:integrability}, by Lemma \ref{bound_u} we first observe that 
\begin{align*}
	L(\xi,u_{\xi},\dot{\xi}/\alpha)\geqslant L(\xi,0,\dot{\xi}/\alpha)-K|u_{\xi}|\geqslant\theta_0(0)-c_0-KF_1(t,|x_0-x|/t)
\end{align*}
which gives the lower bound of $L(\xi,u_{\xi},\dot{\xi}/\alpha)$. For the upper bound we will treat two cases:
\begin{enumerate}[1.]
  \item Suppose $\alpha\in[1,3/2]$. Then Lemma \ref{convexity} (b) shows that $L(\xi,u_{\xi},\dot{\xi}/\alpha)\leqslant(\kappa+1)^{-1}L(\xi,u_{\xi},\dot{\xi})+\kappa\cdot(\kappa+1)^{-1}(\overline{\theta}_0(0)+K|u_{\xi}|)$;
  \item For the case $\alpha\in[1/2,1]$, we need condition (L2'). Let $A=\bar{B}(0,tF_2(t,|x_0-x|/t))$ where $F_2$ is determined by Corollary \ref{bound_1} such that $|\xi(s)|\leqslant tF_2(t,|x_0-x|/t)$. Invoking condition (L2') we conclude that
\begin{align*}
	L(\xi,u_{\xi},\dot{\xi}/\alpha)\leqslant&\,L(\xi,0,\dot{\xi}/\alpha)+KF_1(t,|x_0-x|/t)\\
	\leqslant&\,C_A(1+L(\xi,0,\dot{\xi}))+KF_1(t,|x_0-x|/t)\\
	\leqslant&\,C_A(1+L(\xi,u_{\xi},\dot{\xi})+KF_1(t,|x_0-x|/t))+KF_1(t,|x_0-x|/t).
\end{align*}
\end{enumerate}

\medskip

\noindent\textbf{Step II: A necessary condition.} Next, we introduce the family
\begin{equation*}
	\Omega_0=\{\beta:[0,t]\to\R: \beta\in L^{\infty}([0,t]) \ \text{with}\ \int^t_0\beta(s)ds=0\}
\end{equation*}
and let $0\not=\beta\in\Omega_0$. For any $\varepsilon\in \R$ such that
 $|\varepsilon|<\varepsilon_0=\frac 1{4\|\beta\|_{\infty}}<\frac 1{2\|\beta\|_{\infty}}$ we have that $1+\varepsilon\beta\in\Omega$. 
 
Define the functional $\Lambda:\Omega\to\R$ by $\Lambda(\alpha)=u_{\xi,\alpha}(t)$. Since $\Lambda(1+\varepsilon\beta)\geqslant\Lambda(1)$ for  $|\varepsilon|<\frac 1{2\|\beta\|_{\infty}}$, then we have that $\frac d{d\varepsilon}\Lambda(1+\varepsilon\beta)\vert_{\varepsilon=0}=0$ if the derivative exists. Thus, by \eqref{eq_implicit},
\begin{equation}\label{eq:upper1}
 	\frac {\Lambda(1+\varepsilon\beta)-\Lambda(1)}{\varepsilon}=\int^t_0e^{\int^t_s\widehat{L_u^{\varepsilon}}dr}\lambda_{\varepsilon}(s)\ ds,
\end{equation}
where $\widehat{L_u^{\varepsilon}}=\widehat{L_u^{1+\varepsilon\beta}}$ and
\begin{align*}
	\lambda_{\varepsilon}(s):=&\,\frac{L(\xi,u_{\xi},\dot{\xi}/(1+\varepsilon\beta))(1+\varepsilon\beta)-L(\xi,u_{\xi},\dot{\xi})}{\varepsilon}\\
	=&\,L(\xi,u_{\xi},\dot{\xi}/(1+\varepsilon\beta))\cdot \beta+\frac 1{\varepsilon}(L(\xi,u_{\xi},\dot{\xi}/(1+\varepsilon\beta))-L(\xi,u_{\xi},\dot{\xi})).
\end{align*}
We claim that
\begin{equation}\label{eq:dBR}
	0=\frac d{d\varepsilon}\Lambda(1+\varepsilon\beta)\vert_{\varepsilon=0}=\int^t_0e^{\int^t_sL_udr}\,\left\{L(\xi,u_{\xi},\dot{\xi})-L_v(\xi,u_{\xi},\dot{\xi})\cdot\dot{\xi}\right\}\,\beta\ ds.
\end{equation}

\medskip

\noindent\textbf{Step III: On the summability.} Set
\begin{align*}
	l_{\varepsilon}(s):=L_v(\xi,u_{\xi},\dot{\xi}/(1+\varepsilon\beta))\cdot\dot{\xi}/(1+\varepsilon\beta)-L(\xi,u_{\xi},\dot{\xi}/(1+\varepsilon\beta)).
\end{align*}
Notice that we take out the variable $s$ on right side of the inequalities above. In view of Lemma \ref{convexity} (a) and Lemma \ref{bound_u} we have that $l_{\varepsilon}(s)$ is bounded below by $-(\overline{\theta}_0(0)+KF_1(t,|x_0-x|/t))$. 
%
By convexity we have that
\begin{align*}
	&\,L(\xi,u_{\xi},\dot{\xi}/(1+\varepsilon\beta))-L(\xi,u_{\xi},\dot{\xi})\leqslant L_v(\xi,u_{\xi},\dot{\xi}/(1+\varepsilon\beta))\cdot\{\dot{\xi}-\dot{\xi}/(1+\varepsilon\beta)\}\\
	=&\,-\varepsilon\beta\cdot L_v(\xi,u_{\xi},\dot{\xi}/(1+\varepsilon\beta))\cdot\dot{\xi}/(1+\varepsilon\beta).
\end{align*}
It follows that
\begin{align}
	\lambda_{\varepsilon}(s)\leqslant&\,-\beta(s)\{L_v(\xi,u_{\xi},\dot{\xi}/(1+\varepsilon\beta))\cdot\dot{\xi}/(1+\varepsilon\beta)-L(\xi,u_{\xi},\dot{\xi}/(1+\varepsilon\beta))\}\nonumber\\
	=&\,-\beta(s)\cdot l_{\varepsilon}(s),\label{eq:upper2}
\end{align}


%
Let $\beta\in\Omega_0$ and $0<\varepsilon<\varepsilon_0$. We rewrite $\lambda_{\varepsilon}(s)$, $l_{\varepsilon}(s)$ as $\lambda_{\varepsilon}^{\beta}(s)$, $l_{\varepsilon}^{\beta}(s)$ respectively. Set $\beta^+=\beta\cdot\mathbbm{1}_{\{\beta\geqslant0\}}$ and $\beta^-=-\beta\cdot\mathbbm{1}_{\{\beta<0\}}$, then
\begin{align*}
	\beta=\beta^+-\beta^-,\quad\text{and}\quad\beta^{\pm}\geqslant0.
\end{align*}
By \eqref{eq:upper2} we have that
\begin{align*}
	0\leqslant\lambda_{\varepsilon}^{\beta}(s)+\beta^+l(s)_{\varepsilon}^{\beta}(s)\leqslant\beta^-(s)l_{\varepsilon}^{\beta}(s).
\end{align*}
Now observe that $\beta^+(s)l_{\varepsilon}^{\beta}(s)=\beta^+(s)l_{\varepsilon}^{\beta^+}(s)$ and $\beta^-(s)l_{\varepsilon}^{\beta}(s)=\beta^-(s)l_{-\varepsilon}^{\beta^-}(s)$. Then the inequalities above can recast as follows
\begin{equation}\label{eq:integrability2}
	0\leqslant\lambda_{\varepsilon}^{\beta}(s)+\beta^+(s)l_{\varepsilon}^{\beta^+}(s)\leqslant\beta^-(s)l_{-\varepsilon}^{\beta^-}(s).
\end{equation}
Lemma \ref{convexity} (a) ensures that $\varepsilon\mapsto l_{\varepsilon}^{\beta^-}$ is decreasing on $[-\varepsilon_0,\varepsilon_0]$ and we conclude that
\begin{equation}
	\beta^-l_{-\varepsilon}^{\beta^-}\leqslant\beta^-l_{-\varepsilon_0}^{\beta^-}\quad\forall\varepsilon\in(0,\varepsilon_0).
\end{equation}
By Lemma \ref{convexity} (b), we obtain that
\begin{align*}
	\beta^-l_{-\varepsilon}^{\beta^-}=&\,\beta^-\{L_v(\xi,u_{\xi},\dot{\xi}/(1-\varepsilon\beta^-))\cdot\dot{\xi}/(1-\varepsilon\beta^-)-L(\xi,u_{\xi},\dot{\xi}/(1-\varepsilon\beta^-))\}\\
	\leqslant&\,\beta^-(\kappa_{\varepsilon}^{\beta^-})^{-1}L(\xi,u_{\xi},\dot{\xi}/(1-\varepsilon_0\beta^-))-\beta^-((\kappa_{\varepsilon}^{\beta^-})^{-1}+1)L(\xi,u_{\xi},\dot{\xi}/(1-\varepsilon\beta^-))
\end{align*}
where $(\kappa_{\varepsilon}^{\beta^-})^{-1}=\frac{1-\varepsilon_0\beta^-}{\varepsilon_0-\varepsilon}\cdot(\beta^-)^{-1}$. In view of \eqref{eq:integrability}, \eqref{eq:integrability2} and the fact that $\beta^-(\kappa_{\varepsilon}^{\beta^-})^{-1}$ is bounded, we conclude that $\beta^-l_{-\varepsilon}^{\beta^-}\in L^1([0,t])$ for all $\varepsilon\in(0,\varepsilon_0]$.

\medskip

\noindent\textbf{Step IV: Erdmann condition.} Thus integrating \eqref{eq:integrability2} and by Lebesgue's theorem we obtain that
\begin{align*}
	0\leqslant\int^t_0e^{\int^t_sL_udr}l_0(s)\beta^+(s)ds\leqslant\int^t_0e^{\int^t_sL_udr}l_0(s)\beta^-(s)ds.
\end{align*}
Therefore, $\int^t_0e^{\int^t_sL_udr}l_0(s)\beta(s)ds\leqslant0$ and \eqref{eq:dBR} follows since $\beta\in\Omega_0$ is arbitrary.

Now, observe that  the primitive $\mu(s):=\int^s_0\beta(r)dr$ gives a one-to-one correspondence between  $\Omega_0$
and the set
\begin{align*}
	\Omega_1=\{\mu:[0,t]\to\R: \mu\ \text{is Lipschitz continuous with}\ \mu(0)=\mu(t)=0\}.
\end{align*}
Thus, \eqref{eq:dBR} can recast as follows 
\begin{equation*}
0=-e^{\int^t_0L_udr}\,\int^t_0E(s)\mu'(s)\ ds\quad\forall \mu\in\Omega_1.
\end{equation*}
Recalling $E(s)=-e^{-\int^s_0L_udr}l_0(s)\in L^1([0,t])$ by Step III, then a basic lemma in the  calculus of variations ensures that 
$E(s)$ is constant on $[0,t]$. 
\end{proof}
\begin{proposition}\label{Lip}
Suppose \mbox{\rm (L1), (L2')} and \mbox{\rm (L3)} are satisfied. Let $u_0\in\R$  and $R>0$ be fixed. Then there exists  a continuous function $F=F_{u_0,R}:[0,+\infty)\times[0,+\infty)\to[0,+\infty)$,
	with $F(t,r)$ nondecreasing in both variables and superlinear with respect to $r$, such that for any given $t>0$ and $x,x_0\in\R^n$, with 
	$|x-x_0|\leqslant R$, every minimizer  $\xi\in\Gamma^t_{x_0,x}$ for \eqref{eq:app_fundamental_solution} satisfies 
	\begin{align*}
		\operatorname*{ess\ sup}_{s\in[0,t]}|\dot{\xi}(s)|\leqslant F(t,R/t).
	\end{align*}
\end{proposition}
\begin{proof}
Let $\xi\in\Gamma^t_{x_0,x}$ be as above. Invoking Lemma \ref{Erdmann_condition}, $E$ is constant on a subset of $[0,t]$ of full measure. Let 
\begin{align*}
	E_1(s)=L_v(\xi(s),u_{\xi}(s),\dot{\xi}(s))\cdot\dot{\xi}(s)-L(\xi(s),u_{\xi}(s),\dot{\xi}(s)).
\end{align*}
Set $l_{\xi}(s,\alpha)=L(\xi(s),u_{\xi}(s),\dot{\xi}(s)/\alpha)\alpha$ for all $s\in[0,t]$ and $\alpha>0$. A simple computation shows that $l_{\xi}(s,\alpha)$ is convex in $\alpha$ and 
\begin{align*}
	\frac{d}{d\alpha}l_{\xi}(s,\alpha)\vert_{\alpha=1}=-E_1(s).
\end{align*}
Taking $s_0\in[0,t]$ such that $|\dot{\xi}(s_0)|=\operatorname*{ess\ inf}_{s\in[0,t]}|\dot{\xi}(s)|$, by convexity we have that
\begin{align*}
	-E_1(s_0)\geqslant\sup_{\alpha<1}\frac{l_{\xi}(s_0,\alpha)-l_{\xi}(s_0,1)}{\alpha-1}
\end{align*}
Let us now take, in the above inequality, $\alpha=3/4$. Then,  by (L2), (L3), Lemma \ref{bound_u} and Corollary \ref{bound_1} we conclude that
\begin{align*}
	-E_1(s_0)\geqslant&\,4(l_{\xi}(s_0,1)-l_{\xi}(s_0,3/4))=4(L(\xi(s_0),u_{\xi}(s_0),\dot{\xi}(s_0))-l_{\xi}(s_0,3/4))\\
		\geqslant&\,-4c_0-4KF_1(t,R/t)-3L(\xi(s_0),u_{\xi}(s_0),\frac 43\dot{\xi}(s_0)) \\
		\geqslant&\,-4c_0-4KF_1(t,R/t)-3KF_1(t,R/t)-3L(\xi(s_0),0,\frac 43\dot{\xi}(s_0))\\
		\geqslant&\,-4c_0-7KF_1(t,R/t)-\overline{\theta}_0(\frac 43|\dot{\xi}(s_0)|)\\
		\geqslant&\,-4c_0-7KF_1(t,R/t)-\overline{\theta}_0(F_2(t,R/t)):=-F_3(t,R/t).
\end{align*}
It follows that, for almost all $s\in[0,t]$,
\begin{align*}
	E(s)=E(s_0)=e^{-\int^{s_0}_0L_ud\tau}E_1(s_0)\leqslant e^{Kt}F_3(t,R/t),
\end{align*}
and
\begin{equation}\label{eq:Erdmann_condition}
	E_1(s)=e^{\int^s_0L_ud\tau}E(s)\leqslant e^{2Kt}F_3(t,R/t):=F_4(t,R/t).
\end{equation}

Now, let $s$ be such that $\dot{\xi}(s)$ exists and \eqref{eq:Erdmann_condition} holds. By convexity, we have that
\begin{align*}
	&\,L(\xi(s),u_{\xi}(s),\dot{\xi}(s)/(1+|\dot{\xi}(s)|))-L(\xi(s),u_{\xi}(s),\dot{\xi}(s))\\
	\geqslant&\,((1+|\dot{\xi}(s)|)^{-1}-1)\cdot\langle L_v(\xi(s),u_{\xi}(s),\dot{\xi}(s)),\dot{\xi}(s)\rangle\\
	\geqslant&\,((1+|\dot{\xi}(s)|)^{-1}-1)\cdot(L(\xi(s),u_{\xi}(s),\dot{\xi}(s))+F_4(t,R/t)).
\end{align*}
It follows that
\begin{multline*}
	L(\xi(s),u_{\xi}(s),\dot{\xi}(s))
	\\
	\leqslant\, L(\xi(s),u_{\xi}(s),\dot{\xi}(s)/(1+|\dot{\xi}(s)|))(1+|\dot{\xi}(s)|)+F_4(t,R/t)|\dot{\xi}(s)|.
\end{multline*}
Let $C=\sup_{s\in[0,t],|v|\leqslant1}L(\xi(s),u_{\xi}(s),v)$ and by (L2). By Lemma \ref{bound_u} we have that
\begin{equation*}
	C\leqslant\sup_{s\in[0,t],|v|\leqslant1}\{L(\xi(s),0,v)+K|u_{\xi}(s)|\}\leqslant\overline{\theta}_0(1)+KF_1(t,R/t):=F_5(t,R/t).
\end{equation*}
It follows that
\begin{align*}
	L(\xi(s),u_{\xi}(s),\dot{\xi}(s))\leqslant F_5(t,R/t)+(F_5(t,R/t)+F_4(t,R/t))|\dot{\xi}(s)|.
\end{align*}
Therefore, invoking Lemma \ref{bound_u}, we obtain 
\begin{align*}
	&\,(F_5(t,R/t)+F_4(t,R/t)+1)|\dot{\xi}(s)|-(\theta_0^*(F_5(t,R/t)+F_4(t,R/t)+1)+c_0)\\
	\leqslant&\, \theta_0(|\dot{\xi}(s)|)-c_0\leqslant L(\xi(s),0,\dot{\xi}(s))\leqslant L(\xi(s),u_{\xi}(s),\dot{\xi}(s))+K|u_{\xi}(s)|\\
	\leqslant&\, F_5(t,R/t)+(F_5(t,R/t)+F_4(t,R/t))|\dot{\xi}(s)|+KF_1(t,R/t).
\end{align*}
This leads to
\begin{align*}
	|\dot{\xi}(s)|
	&\leqslant(\theta_0^*(F_5(t,R/t)+F_4(t,R/t)+1)+c_0)+F_5(t,R/t)+KF_1(t,R/t)
	\\
	&:=F_6(t,R/t),
\end{align*}
which completes the proof.
\end{proof}

\begin{proposition}
Suppose $L_{\lambda}(x,r,v)=L_0(x,v)-\lambda r$, $r\in\R$, where $L_0$ is a Tonelli Lagrangian. Then $L_{\lambda}$ satisfies condition \mbox{\rm (L1), (L2)} and \mbox{\rm (L3)}. Moreover, the Lipschitz estimate in Proposition \ref{Lip} holds.
\end{proposition}

\begin{proof}
	By solving the Carath\'eodory equation \eqref{eq:caratheodory_L_intro}, we have that
	\begin{align*}
		u_{\xi}(t)=e^{-\lambda t}u_0+e^{-\lambda t}\int^t_0e^{\lambda s}L_0(\xi,\dot{\xi})\ ds.
	\end{align*}
	Therefore problem \eqref{eq:M} is essentially a basic problem in the calculus of variations with a time-dependent Lagrangian $G(t,x,v)=e^{\lambda t}L_0(x,v)$. Moreover, $G$ satisfies a restricted growth condition
	\begin{align*}
		G_t(t,x,v)=\lambda G(t,x,v).
	\end{align*}
	Then any minimizer $\xi$ of \eqref{eq:M} is Lipschitz continuous (see, for instance, \cite[Theorem 4.9]{Buttazzo-Giaquinta-Hildebrandt}). Therefore, Erdmann condition in Lemma \ref{Erdmann_condition} holds with a slight modification of the proof and the expected Lipschitz estimate can obtained as in the proof of Proposition \ref{Lip} similarly.
\end{proof}

\subsection{Regularity of minimizers - Herglotz equations - Lie equations}\label{app_EL}
Let $\xi\in\Gamma^t_{x_0,x}$ be a minimizer of \eqref{eq:app_fundamental_solution} where $u_{\xi}$ is determined uniquely by \eqref{eq:app_caratheodory_L}. For any $\lambda\in\R$ and any Lipschitz function $\eta\in\Gamma^t_{0,0}$,  we denote $\xi_{\lambda}(s)=\xi(s)+\lambda\eta(s)$. It is clear that $\xi_{\lambda}\in\Gamma^t_{x_0,x}$ and $J(\xi)\leqslant J(\xi_{\lambda})$. Let $u_{\xi_{\lambda}}$ be the associated unique solution of \eqref{eq:app_caratheodory_L} with respect to $\xi_{\lambda}$ and the initial condition $u_0$. Notice that
\begin{align*}
	\frac{\partial}{\partial\lambda}J(\xi_{\lambda})\vert_{\lambda=0}=\frac{\partial}{\partial\lambda}u_{\xi_{\lambda}}(t)\vert_{\lambda=0}=0.
\end{align*}
Now for any $s\in[0,t]$ we set
\begin{align*}
	\Delta_{\lambda}(s)=\frac {u_{\xi_{\lambda}}(s)-u_{\xi}(s)}{\lambda}=\frac 1{\lambda}\int^s_0L(\xi_{\lambda},u_{\xi_{\lambda}},\dot{\xi}_{\lambda})-L(\xi,u_{\xi},\dot{\xi})\ d\tau,
\end{align*}
and
\begin{align*}
	f^{\lambda}_1(s)=&\,\frac 1{\lambda}\int^s_0L(\xi_{\lambda},u_{\xi_{\lambda}},\dot{\xi}_{\lambda}(s))-L(\xi_{\lambda},u_{\xi_{\lambda}},\dot{\xi})\ d\tau,\\
	f^{\lambda}_2(s)=&\,\frac 1{\lambda}\int^s_0L(\xi_{\lambda},u_{\xi_{\lambda}},\dot{\xi})-L(\xi,u_{\xi_{\lambda}},\dot{\xi})\ d\tau.
\end{align*}
Then $f^{\lambda}_1$ and $f^{\lambda}_2$ are all absolutely continuous functions on $[0,t]$, and it follows
\begin{align*}
	\Delta_{\lambda}(s)=f^{\lambda}_1(s)+f^{\lambda}_2(s)+\frac 1{\lambda}\int^s_0\widehat{L^{\lambda}_u}\cdot (u_{\xi_{\lambda}}-u_{\xi})\ d\tau,\quad s\in[0,t],
\end{align*}
where
\begin{align*}
	\widehat{L^{\lambda}_u}(\tau)=\int^1_0L_u(\xi(\tau),u_{\xi}(\tau)+\theta(u_{\xi_{\lambda}}(\tau)-u_{\xi}(\tau)),\dot{\xi}(\tau))\ d\theta,\quad \tau\in[0,t].
\end{align*}
Thus, we conclude that for almost all $s\in[0,t]$, the following Carath\'eodory equation holds:
\begin{equation}\label{eq:Delta_lambda}
	\dot{\Delta}_{\lambda}(s)=\dot{f}^{\lambda}_1(s)+\dot{f}^{\lambda}_2(s)+\widehat{L^{\lambda}_u}(s)\cdot \Delta_{\lambda}(s)
\end{equation}
with initial condition $\Delta_{\lambda}(t)=a_{\lambda}$. Notice that $\lim_{\lambda\to0}\Delta_{\lambda}(t)$ exists and $\lim_{\lambda\to0}\Delta_{\lambda}(t)=\lim_{\lambda\to0}a_{\lambda}=\frac{\partial}{\partial\lambda}u_{\xi_{\lambda}}(t)\vert_{\lambda=0}=0$ since $\xi$ is a minimizer of $J$. It is not difficult to solve \eqref{eq:Delta_lambda}, we obtain that
\begin{align*}
	\Delta_{\lambda}(s)=a_{\lambda}e^{\int^s_t\widehat{L^{\lambda}_u}(r)\ dr}+e^{\int^s_t\widehat{L^{\lambda}_u}(r)\ dr}\cdot\int^s_te^{-\int^r_t\widehat{L^{\lambda}_u}(\tau)\ d\tau}\cdot(\dot{f}^{\lambda}_1(r)+\dot{f}^{\lambda}_2(r))\ dr.
\end{align*}
Since $(\xi_{\lambda}(s),\dot{\xi}_{\lambda}(s),u_{\xi_{\lambda}}(s))$ tends $(\xi(s),\dot{\xi}(s),u_{\xi}(s))$ as $\lambda\to0$ for almost all $s\in[0,t]$,  together with Proposition \ref{Lip} and Corollary \ref{equi_int_u}, it follows that, for all $s\in[0,t]$, we have
\begin{equation}\label{eq:EL_4}
	f(s):=\frac{\partial}{\partial\lambda}u_{\xi_{\lambda}}(s)\vert_{\lambda=0}=e^{\int^s_th(r)\ dr}\cdot\int^s_te^{-\int^r_th(\tau)\ d\tau}\cdot g(r)\ dr,\quad f(t)=0,
\end{equation}
where $g=L_x\cdot\eta+L_v\cdot\dot{\eta}$ and $h=L_u$ which are both measurable and bounded. Notice that \eqref{eq:Delta_lambda} implies that
\begin{align*}
	f(s)=\int^s_0g(r)+h(r)f(r)\ dr,\quad s\in[0,t].
\end{align*}
Then, invoking \eqref{eq:EL_4}, we conclude that
\begin{align*}
	0=&\,\int^t_0g(s)+h(s)\cdot e^{\int^s_th(r)\ dr}\cdot\left\{\int^s_te^{-\int^r_th(\tau)\ d\tau}\cdot g(r)\ dr\right\}ds\\
	=&\,\int^t_0g(s)\ ds+e^{\int^s_th(r)\ dr}\cdot\left\{\int^s_te^{-\int^r_th(\tau)\ d\tau}\cdot g(r)\ dr\right\}\bigg\vert^t_{0}\\
	&\,-\int^t_0e^{\int^s_th(r)\ dr}\cdot e^{-\int^s_th(r)\ dr}\cdot g(s)\ ds\\
	=&\,e^{\int^0_th(r)\ dr}\cdot\left\{\int^t_0e^{-\int^s_th(r)\ dr}\cdot g(s)\ ds\right\}.
\end{align*}
It follows that
\begin{align*}
	0=\int^t_0e^{-\int^s_th(r)\ dr}\cdot g(s)\ ds=\int^t_0e^{-\int^s_th(r)\ dr}\cdot(L_x\cdot\eta+L_v\cdot\dot{\eta})(s)\ ds.
\end{align*}

Invoking the fundamental lemma in calculus of variation (see, for instance, Lemma 6.1.1 in \cite{Cannarsa-Sinestrari}), we obtain that, for almost all $s\in[0,t]$,
\begin{align*}
	\frac d{ds}e^{-\int^s_th(r)\ dr}L_v(\xi(s),u_{\xi}(s),\dot{\xi}(s))=e^{-\int^s_th(r)\ dr}L_x(\xi(s),u_{\xi}(s),\dot{\xi}(s)).
\end{align*}
This leads to the so called Herglotz equation (for almost all $s\in[0,t]$)
\begin{equation}\label{eq:Herglotz}
	\begin{split}
		&\,\frac d{ds}L_v(\xi(s),u_{\xi}(s),\dot{\xi}(s))\\
	=&\,L_x(\xi(s),u_{\xi}(s),\dot{\xi}(s))+L_u(\xi(s),u_{\xi}(s),\dot{\xi}(s))L_v(\xi(s),u_{\xi}(s),\dot{\xi}(s)).
	\end{split}
\end{equation}

Since $L$ is of class $C^2$ and $L(x,u,\cdot)$ is strictly convex, then by the standard argument as in \cite[Section 6.2]{Cannarsa-Sinestrari}, we conclude that:

\begin{theorem}\label{Herglotz_contact}
Under our standing assumptions, we have the following regularity properties for any minimizer $\xi$ for \eqref{eq:app_fundamental_solution}:
\begin{enumerate}[\rm (1)]
  \item Both $\xi$ and $u_{\xi}$ are of class $C^2$ and $\xi$ satisfies Herglotz equation \eqref{eq:Herglotz} for all $s\in[0,t]$ where $u_{\xi}$ is the unique solution of \eqref{eq:app_caratheodory_L};
  \item Let $p(s)=L_v(\xi(s),u_{\xi}(s),\dot{\xi}(s))$ be the dual arc. Then $p$ is also of class $C^2$ and we conclude that $(\xi,p,u_{\xi})$ satisfies Lie equation \eqref{eq:contact}.
\end{enumerate}
\end{theorem}


\begin{proof}
	We first need to show that $\xi$ is of class $C^1$. Let $N$ be the set of zero Lebesgue measure where $\dot{\xi}$ does not exist.  For $\bar{t}\in[0,t]$, choose a sequence $\{t_k\}\in [0,T]\setminus N$ such that $t_k\to\bar{t}$. Then $\dot{\xi}(t_k)\to\bar{v}$ for some $\bar{v}\in\R^n$ (up to subsequences) and
	\begin{align*}
		&\,L_v(\xi(\bar{t}),u_{\xi}(\bar{t}),\dot{\xi}(\bar{t}))=\lim_{k\to\infty}L_v(\xi(t_k),u_{\xi}(t_k),\dot{\xi}(t_k))\\
		=&\,\int^{t}_0\{L_x(\xi(s),u_{\xi}(s),\dot{\xi}(s))+L_u(\xi(s),u_{\xi}(s),\dot{\xi}(s))L_v(\xi(s),u_{\xi}(s),\dot{\xi}(s))\} ds
	\end{align*}
	by \eqref{eq:Herglotz}. From the strict convexity of $L$ it follows that the map $v\mapsto L_v(\xi(s),u_{\xi}(s),v)$ is a diffeomorphism. This implies that $\bar{v}$ is uniquely determined, i.e.,
	\begin{align*}
		\lim_{[0,t]\setminus N\ni s\to\bar{t}}\dot{\xi}(s)=\bar{v}.
	\end{align*}
	Now, by Lemma 6.2.6 in \cite{Cannarsa-Sinestrari}, $\dot{\xi}(\bar{t})$ exists and $\lim_{[0,t]\setminus N\ni s\to\bar{t}}\dot{\xi}(s)=\dot{\xi}(\bar{t})$. It follows $\xi$ is of class $C^1$. In view of \eqref{eq:app_caratheodory_L}, $u_{\xi}$ is also of class $C^1$.
	
	In view of \eqref{eq:Herglotz1}, by setting
	\begin{align*}
		F(s)=\int^{s}_0\{L_x(\xi,u_{\xi},\dot{\xi})+L_u(\xi,u_{\xi},\dot{\xi})L_v(\xi,u_{\xi},\dot{\xi})\}\ d\tau,
	\end{align*}
	we have that
	\begin{align*}
		\{L_v(\xi(s),u_{\xi}(s),v)-F(s)\}\vert_{v=\dot{\xi}(s)}=L_v(\xi(0),u_{\xi}(0),\dot{\xi}(0)).
	\end{align*}
	Then, the implicit function theorem implies $\dot{\xi}$ is of class $C^1$ since both $F$ and $L_v$ are of class $C^1$. Therefore we conclude that $\xi$ is of class $C^2$ and $u_{\xi}$ is of class $C^2$ by \eqref{eq:app_caratheodory_L}.
	
	The rest part of the proof is standard and we omit it.
\end{proof}

\section{Concluding remarks}

\subsection{Equivalence of Herglotz' variational principle and the implicit variational principle}

In the recent work \cite{Wang-Wang-Yan1,Wang-Wang-Yan2}, the authors introduce an implicit variational principle on closed manifolds which is essentially equivalent to Herglotz' principle.

\begin{proposition}[\cite{Wang-Wang-Yan1}]\label{WWY}
Let $M$ be a $C^2$ closed manifold and let $L:TM\to\R$ be of class $C^3$ and satisfy conditions (L1)-(L3) for $M$ instead of $\R^n$ here.	Given any $x_0\in M$ and $u_0\in\R$, there exists a (unique) continuous function $h_{x_0,u_0}(t,x)$ defined on $(0,+\infty)\times M$ satisfying
\begin{align*}
	h_{x_0,u_0}(t,x)=u_0+\inf_{\xi}\int^t_0L(\xi(s),h_{x_0,u_0}(s,\xi(s)),\dot{\xi}(s))\ ds,
\end{align*}
where $\xi$ is taken over all the Lipschitz continuous curves on $M$ connecting $\xi(0)=x_0$ and $\xi(t)=x$.

Moreover, let $\xi$ be any curve achieving the infimum together with the curves $p$ and $u$ defined by
\begin{align*}
	u(s)=h_{x_0,u_0}(s,\xi(s)),\quad p(s)=L_v(\xi(s),u(s),\dot{\xi}(s)).
\end{align*}
Then $(\xi,p,u_{\xi})$ is a solution of \eqref{eq:contact} with conditions $\xi(0)=x_0$, $\xi(t)=x$ and $\lim_{s\to0^+}u(s)=u_0$.
\end{proposition}

\begin{proposition}\label{eq:dyn_pro}
	Let $x_0,x\in\R^n$, $t>0$ and $u_0\in\R$. For any $\xi\in\Gamma^t_{x_0,x}$ being a minimizer of \eqref{eq:app_fundamental_solution}, we denote by $u_{\xi}(s,u_0)$ the unique solution of \eqref{eq:app_caratheodory_L} with $u_{\xi}(0,u_0)=u_0$. Then, for any $0<t'<t$, the restriction of $\xi$ on $[0,t']$ is a minimizer for
  $$
  A(t',x_0,x,u_0):=u_0+\inf\int^{t'}_0L(\xi(s),u_{\xi}(s),\dot{\xi}(s))\ ds
  $$
  with $u_{\xi}$ the unique solution of \eqref{eq:app_caratheodory_L} restricted on $[0,t']$. Moreover,
  \begin{equation}\label{eq:implcite_2}
  	A(s,x_0,\xi(s),u_0)=u_{\xi}(s,u_0),\quad\forall s\in[0,t],
  \end{equation}
  and $A(s_1+s_2,x_0,\xi(s_1+s_2),u_0)=A(s_2,\xi(s_1),\xi(s_1+s_2),u_{\xi}(s_1))$ for any $s_1,s_2>0$ and $s_1+s_2\leqslant t$.

  In particular, if $h_{x_0,u_0}$ is from Proposition \ref{WWY}, then
  \begin{equation}\label{eq:equiv}
  	u_{\xi}(s)=h_{x_0,u_0}(s,\xi(s)),\quad s\in[0,t].
  \end{equation}
\end{proposition}

\begin{remark}
The relation \eqref{eq:equiv} holds only when $\xi$ is a minimizer of \eqref{eq:app_fundamental_solution}. 	
\end{remark}

\begin{proof}
		Suppose $x_0,x\in\R^n$, $t>0$ and $u_0\in\R$. Let $\xi\in\Gamma^t_{x_0,x}$ be a minimizer of \eqref{eq:app_fundamental_solution} and $u_{\xi}(s)=u_{\xi}(s,u_0)$ be the unique solution of \eqref{eq:app_caratheodory_L} with $u_{\xi}(0)=u_0$.
	
	Now, let $0<t'<t$. Let $\xi_1\in\Gamma^{t'}_{x,\xi(t')}$ and $\xi_2\in\Gamma^{t-t'}_{\xi(t'),y}$ be the restriction of $\xi$ on $[0,t']$ and $[t',t]$ respectively. Then, we have that
	\begin{align*}
		u_{\xi}(t',u_0)=&\,u_0+\int^{t'}_0L(\xi_1(s),u_{\xi_1}(s),\dot{\xi_1}(s))\ ds,\\
		u_{\xi}(t,u_0)-u_{\xi}(t',u_0)=&\,\int^t_{t'}L(\xi_2(s),u_{\xi_2}(s),\dot{\xi_2}(s))\ ds.
	\end{align*}
	Then both $\xi_1$ and $\xi_2$ are minimal curve for \eqref{eq:app_fundamental_solution} restricted on $[0,t']$ and $[t',t]$ respectively by summing up the equalities above and the assumption that $\xi$ is a minimizer of \eqref{eq:app_fundamental_solution}. In particular, \eqref{eq:implcite_2} follows. The next assertion is direct from the relation
	$$
	u_{\xi}(s_1+s_2,u_0)=u_{\xi}(s_2,u_{\xi}(s_1)),\quad \forall s_1,s_2>0,\ s_1+s_2\leqslant t,
	$$
	since $u_{\xi}$ solves \eqref{eq:app_caratheodory_L}. The last assertion is a direct application of Gronwall's inequality. Indeed, we know that for all $s\in[0,t]$,
	\begin{align*}
		u_{\xi}(s)=&\,u_0+\int^s_0L(\xi(r),u_{\xi}(r),\dot{\xi(r)})\ dr,\\
		h_{x_0,u_0}(s,\xi(s))=&\,u_0+\int^s_0L(\xi(r),h_{x_0,u_0}(r,\xi(r)),\dot{\xi}(r))\ dr.
	\end{align*}
	By condition (L3), it follows that
	\begin{align*}
		|h_{x_0,u_0}(s,\xi(s))-u_{\xi}(s)|\leqslant K\int^s_0|h_{x_0,u_0}(r,\xi(r))-u_{\xi}(r)|\ dr.
	\end{align*}
	Our conclusion is a consequence of Gronwall's inequality.
\end{proof}

\subsection{Herglotz' generalized variational principle on manifolds}

Now, we try to explain how to move the Herglotz' generalized variational principle to any connected and closed smooth manifold $M$.

Fix $x,y\in M$, $t>0$ and $u\in\R$. Let $\xi\in\Gamma^t_{x,y}(M)$, we consider the Carath\'eodory equation
\begin{equation}\label{eq:app_caratheodory_L2}
	\begin{cases}
		\dot{u}_{\xi}(s)=L(\xi(s),u_{\xi}(s),\dot{\xi}(s)),\quad a.e.\ s\in[0,t],&\\
		u_{\xi}(0)=u.&
	\end{cases}
\end{equation}
We define the action functional
\begin{equation}\label{eq:app_fundamental_solution2}
	J(\xi):=\int^t_0L(\xi(s),u_{\xi}(s),\dot{\xi}(s))\ ds,
\end{equation}
where $\xi\in\Gamma^t_{x,y}(M)$ and $u_{\xi}$ is defined in \eqref{eq:app_caratheodory_L}. Our purpose is to minimize $J(\xi)$ over
\begin{align*}
	\mathcal{A}(M)=\{\xi\in\Gamma^t_{x,y}(M): \text{\eqref{eq:app_caratheodory_L2} admits an absolutely continuous solution $u_{\xi}$}\}.
\end{align*}
Notice that $\mathcal{A}(M)\not=\varnothing$ because it contains all piecewise $C^1$ curves connecting $x$ to $y$. In view of the remark before Lemma \ref{u_low_bound}, for each $a\in\R$,
\begin{align*}
	\mathcal{A}(M)=\{\xi\in\Gamma^t_{x,y}(M): \text{the function $s\mapsto L(\xi(s),a,\dot{\xi}(s))$ belongs to $L^1([0,t])$}\}.
\end{align*}

We begin with the case when $M=\R^n$. Fix $\kappa>0$. Suppose $0<t\leqslant 1$, $x,y\in\R^n$ such that $|x-y|\leqslant\kappa t$. Suppose $\eta\in\mathcal{A}(\R^n)$ is a minimizer of the action functional $\eta\mapsto J(\eta)$. Invoking the aforementioned {\em a priori} estimates, $\eta$ is as smooth as $L$. Moreover, there exist constants $C_1(\kappa)>0,C_2(u,t,\kappa)>0$ such that
\begin{align*}
	\sup_{s\in[0,t]}|\dot{\eta}(s)|\leqslant C_1(\kappa),\quad \sup_{s\in[0,t]}|\eta(s)-x|\leqslant C_1(\kappa)t,\quad \sup_{s\in[0,t]}|u_{\eta}(s)|\leqslant C_2(u,t,\kappa).
\end{align*}
Let $D_1=B_{\R^n}(x,\kappa t)$ and $D_2=B_{\R^n}(x,(C_1(\kappa)+1)t)$, where the subscript is used for the ball in $\R^n$. Then, $D_1\subset D_2$ since $\kappa\leqslant C_1(\kappa)$.
By denoting
\begin{align*}
	\mathcal{B}(\R^n)=\{\eta\in\mathcal{A}(\R^n): \text{$\eta(s)\in D_2$ for all $s\in[0,t]$}\}.
\end{align*}
Therefore we can claim that for any $x\in\R^n$ and $y\in D_1$ the following problems are equivalent:
\begin{align*}
	\inf_{\mathcal{A}(\R^n)}J(\xi)=\inf_{\mathcal{B}(\R^n)}J(\xi)
\end{align*}
They admit the same minimizers.

Now we move to the manifold case. Let $\{(B_i,\Phi_i)\}$ be a local chart for the $C^2$ closed manifold $M$. We can suppose that $\{B_i\}_{i=1}^N$ is a finite open cover of $M$ and $\Phi_i:B_i\to D_2\subset\R^n$ is a $C^2$-diffeomorphism for each $i=1,\ldots,N$ and $\Phi_j^{-1}\circ\Phi_i:B_i\cap B_j\to B_i\cap B_j$ is a $C^2$-diffeomorphism for each $i\not=j=1,\ldots,N$.

Fix $i$, let $B=B_i$ and $\Phi=\Phi_i:B\to D_2$ be a local coordinate. Let $L:TM\times\R\to\R$ be a Lagrangian satisfying (L1)-(L3). Then
\begin{align*}
	(\Phi,d\Phi):TB\to D_2\times\R^n
\end{align*}
defines a local trivilazation of $TB$. Let $L_{\Phi}:D_2\times\R^n\times\R\to\R$ be defined by
\begin{align*}
	L_{\Phi}(\bar{x},u,\bar{v})=L(\Phi^{-1}(\bar{x}),u,d\Phi^{-1}(\bar{x})\bar{v}),\quad (\bar{x},\bar{v})\in D_2\times\R^n,\ u\in\R.
\end{align*}
Therefore, Herglotz' generalized variational principle for $L$ restricted to $TB\times\R$ is equivalent to the one for $L_{\Phi}$ on $D_2\times\R^n\times\R\to\R$ since $\Phi$ is a bi-Lipschitz homeomorphism and a $C^2$-diffeomorphism.

\begin{proposition}\label{Main_Hergolotz}
	Fix $\kappa>0$, $0<t\leqslant 1$. Then there exist a local chart $\{(B_i,\Phi_i)\}_{i=1}^N$ and a constant $C_2(\kappa)>0$ such that each $B_i\subset B_M(x,C_2(\kappa)t)$, and for any $x,y\in B_i$ and $u\in\R$, the following points on the Hergolotz's generalized variational principle hold:
	\begin{enumerate}[\rm (a)]
	\item The functional
	\begin{align*}
		\mathcal{A}(B_i)\ni\xi\mapsto J(\xi)=\int^t_0L(\xi(s),u_{\xi}(s),\dot{\xi}(s))\ ds,
	\end{align*}
	where $u_{\xi}$ is determined by \eqref{eq:app_caratheodory_L} with initial condition $u_{\xi}(0)=u$, admits a minimizer in $\mathcal{A}(M)$.
	\item Suppose $x,y\in B_i$. Let $\xi\in\mathcal{A}(B_i)$ be a minimizer of $J$. Then there exists a function $F=F_{B_i}:[0,+\infty)\times[0,+\infty)\to[0,+\infty)$, with $F(\cdot,r)$ being nondecreasing for any $r\geqslant0$, such that
	\begin{equation*}
		|u_{\xi}(s)|\leqslant tF(t,\kappa)+C(t)|u|,\quad s\in[0,t]
	\end{equation*}
	where $C(t)>0$ is also nondecreasing in $t$.
	\item Suppose $x,y\in B_i$. Let $\xi\in\mathcal{A}(B_i)$ be a minimizer of $J$. Then, there exists a function $F=F_{u,B_i}:[0,+\infty)\times[0,+\infty)\to[0,+\infty)$, with $F(\cdot,r)$ is nondecreasing for any $r\geqslant0$, such that
	\begin{align*}
		\operatorname*{ess\ sup}_{s\in[0,t]}|\dot{\xi}(s)|\leqslant F(t,\kappa).
	\end{align*}
	\item We have the following regularity properties for any minimizer $\xi$ for \eqref{eq:app_fundamental_solution2}:
	\begin{enumerate}[\rm 1)]
	\item Both $\xi$ and $u_{\xi}$ are of class $C^2$ and $\xi$ satisfies Herglotz equation \eqref{eq:Herglotz} in local charts for all $s\in[0,t]$ where $u_{\xi}$ is the unique solution of \eqref{eq:app_caratheodory_L};
	\item Let $p(s)=L_v(\xi(s),u_{\xi}(s),\dot{\xi}(s))$ be the dual arc. Then $p$ is also of class $C^2$ and we conclude that $(\xi,p,u_{\xi})$ satisfies Lie equation \eqref{eq:contact} in local charts for all $s\in[0,t]$.
\end{enumerate}
\end{enumerate}
\end{proposition}

Thus, by using $L_{\Phi}$, it is not difficult to see that there exists a finer open cover, which we also denote by $\{(B_i,\Phi_i)\}_{i=1}^N$, such that the Herglotz' generalized variational principle can be applied in the case when $x,y\in B_i$ and $0<t\leqslant 1$ ($i=1,\ldots,N$) since $\{\Phi_i\}_{i=1}^N$ is equi-bi-Lipschitz.

Now, let us recall the standard ``broken geodesic'' argument. Pick any $x,y\in M$, $t>0$ and $u\in\R$. Let $\{(B_i,\Phi_i)\}_{i=1}^N$  be the local chart in the proposition above. We suppose without loss of generality that $x\in B_1$ and $y\in B_N$. Let $\xi\in\mathcal{A}(M)$. Then there exists a partition $0=t_0<t_1<t_2<\cdots<t_{k-1}<t_k=t$ such that $z_j=\xi(t_j)$ and $z_{j+1}=\xi(t_{j+1})$ are contained in the same $B_i$. For each $j$, we define
\begin{align*}
	h_L^j(t_{j+1}-t_j,z_j,z_{j+1},u_j)=\inf_{\xi_j}\int^{t_{j+1}}_{t_j}L(\xi_j(s),u_{\xi_j}(s),\dot{\xi}_j(s))\ ds,
\end{align*}
where $\xi_j$ is an absolutely continuous curve constrained in $B_i$ connecting $z_j$ to $z_{j+1}$ and $u_{\xi_j}$ is uniquely determined by \eqref{eq:app_caratheodory_L} with initial condition $u_j$. Now we consider the problem
\begin{equation}\label{eq:conjunction}
	g(t,x,y,u):=\inf\sum^k_{j=1}h_L^j(t_{j+1}-t_j,z_j,z_{j+1},u_j),
\end{equation}
where the infimum is taken over any partition $0=t_0<t_1<t_2<\cdots<t_{k-1}<t_k=t$, $z_j,z_{j+1}\in M$ contained in the same $B_i$ and $u_j\in\R$. Due to Proposition \ref{Main_Hergolotz} (b), $\{u_j\}$ can be constrained in a compact subset of $\R$ depending on $u$, $x,y$ and $t$. Therefore the infimum in \eqref{eq:conjunction} can be attained. Thanks to the local semiconcavity of the fundamental solution $h_L^j$, $h_L^j$ is differentiable at each minimizer which leads to the fact
\begin{align*}
	h_L(t,x,y,u)=g(t,x,y,u).
\end{align*}

\begin{proposition}\label{Main_Hergolotz2}
    Proposition \ref{Main_Hergolotz} holds for any connected and closed $C^2$ manifold $M$.
\end{proposition}

\subsection{Further remarks}

Comparing to the method used in \cite{Wang-Wang-Yan1,Wang-Wang-Yan2}, one can see more from our approach  as follows:
\begin{enumerate}[--]
	\item We can derive the generalized Euler-Lagrange equations in a modern and rigorous way which does not appear in both \cite{Wang-Wang-Yan1,Wang-Wang-Yan2};
	\item There should be an extension of the main results of this paper under much more general conditions (like Osgood type conditions) to guarantee the existence and uniqueness of the solutions of the associated Carath\'eodory equation \eqref{eq:app_caratheodory_L}.
	\item Along this line, the quatitative semiconcavity and convexity estimate of the associated fundamental solutions have been obtained in \cite{Cannarsa-Cheng-Yan} recently, which is useful for the intrinsic study of the global propagation of singularities of the viscosity solutions of \eqref{eq:static_intro} and \eqref{eq:evolutionary_intro} (\cite{Cannarsa-Cheng1,Cannarsa-Cheng3,CCF,CCMW});
	\item When the Lagrangian has the form $L(x,v)-\lambda u$, by solving the associated Carath\'eodory equation \eqref{eq:app_caratheodory_L} directly, one gets the representation formula for the associated viscosity solutions immediately (\cite{DFIZ,Su-Wang-Yan,Zhao-Cheng,Chen-Cheng-Zhang}). The representation formula bridges the PDE aspects of the problem with the dynamical ones;
	\item Consider a family of Lagrangians in the form $\{\mathbf{L}(x,v)+\sum_{i=1}^ka_{ij}u_i\}$, a problem of  Herglotz' variational principle in the vector form is closely connected to certain stochastic model of weakly coupled Hamilton-Jacobi equations (see, for instance, \cite{Davini-Zavidovique,Figalli-Gomes-Marcon,Mitake-Siconolfi-Tran-Yamada}).
\end{enumerate}

\begin{acknowledgement}
This work is partly supported by Natural Scientific Foundation of China (Grant No.11871267, No.11631006, No.11790272 and No.11771283), and the National Group for Mathematical Analysis, Probability and Applications (GNAMPA) of the Italian Istituto Nazionale di Alta Matematica ``Francesco Severi''. Kaizhi Wang is also supported by China Scholarship Council (Grant No. 201706235019). The authors acknowledge the MIUR Excellence Department Project awarded to the Department of Mathematics, University of Rome Tor Vergata, CUP E83C18000100006. The authors are grateful to Qinbo Chen, Cui Chen and Kai Zhao for helpful discussions. This work was motivated when the first two authors visited Fudan University in June 2017. The authors also appreciate the anonymous referee for helpful suggestion to improve the paper.
\end{acknowledgement}

\section*{Appendix}
\addcontentsline{toc}{section}{Appendix}

Let $\Omega\subset\R^{n+1}$ be an open set. A function $f:\Omega\subset\R\times\R^n\to\R^n$ is said to satisfy {\em Carath\'eodory condition} if
\begin{itemize}[-]
  \item for any $x\in\R^n$, $f(\cdot,x)$ is measurable;
  \item for any $t\in\R$, $f(t,\cdot)$ is continuous;
  \item for each compact set $U$ of $\Omega$, there is an integrable function $m_U(t)$ such that
  $$
  |f(t,x)|\leqslant m_U(t),\quad (t,x)\in U.
  $$
\end{itemize}
A classical problem is to find an absolutely continuous function $x$ defined on a real interval $I$ such that $(t,x(t))\in\Omega$ for $t\in I$ and satisfies the following Carath\'eodory equation
\begin{equation}\label{eq:caratheodory_system}
	\dot{x}(t)=f(t,x(t)),\quad a.e., t\in I.
\end{equation}

\begin{proposition}[Carath\'eodory]\label{caratheodory}
	If $\Omega$ is an open set in $\R^{n+1}$ and $f$ satisfies the Carath\'eodory conditions on $\Omega$, then, for any $(t_0, x_0)$ in $\Omega$, there is a solution of \eqref{eq:caratheodory_system} through $(t_0, x_0)$. Moreover, if the function $f(t,x)$ is also locally Lipschitzian in $x$ with a measurable Lipschitz function, then the uniqueness property of the solution remains valid.
\end{proposition}

\noindent For the proof of Proposition \ref{caratheodory} and more results related to Carath\'eodory equation \eqref{eq:caratheodory_system}, the readers can refer to \cite{Coddington_Levinson,Filippov}.


%
%
%

\end{document}